\documentclass[12pt, reqno]{amsart}
\usepackage{amssymb,amscd,amsthm}
\usepackage{latexsym}
\usepackage{amsmath}

\usepackage{pb-diagram}
\usepackage{bbm}
\usepackage{color}
\usepackage[colorlinks]{hyperref}
\usepackage{graphicx}
\usepackage{hyperref}
\hypersetup{
    colorlinks,
    citecolor=blue,
    filecolor=blue,
    linkcolor=blue,
    urlcolor=blue
}
\usepackage{enumerate}
\usepackage{todonotes}

\renewcommand{\qed}{{\nopagebreak \hfill $\dashv$
 \par\bigskip}}

\date{}

\newcommand{\diff}[1]{\mbox{Diff}{#1}}
\newcommand{\gfs}{\emph{Graphs}}
\newcommand{\la}{\langle}
\newcommand{\ra}{\rangle}

\newtheorem{theorem}{Theorem}
\newtheorem*{theorem*}{Theorem}
\newtheorem{remark}[theorem]{Remark}
\newtheorem{lemma}[theorem]{Lemma}
\newtheorem{prop}[theorem]{Proposition}
\newtheorem{coro}[theorem]{Corollary}
\newtheorem{definition}[theorem]{Definition}

\newtheorem{proposition}[theorem]{Proposition}

\newcommand{\RN}[1]{\textup{\uppercase\expandafter{\romannumeral#1}}%
}

\newcommand{\poZ}{\mathbb Z}
\newcommand{\nn}{{\mathbb N}}

\newcommand{\mck}{{\mathcal K}}
\newcommand{\mcb}{\mathcal B}

\newcommand{\bfni}[1]{\noindent {{\bf{#1}}}}

\renewcommand{\qed}{{\nopagebreak \hfill $\dashv$
 \par\bigskip}}

\newcommand{\pf}{{\par\noindent{$\vdash$\ \ \ }}}

\title[Anti-classification for smooth dynamical systems]{Anti-classification results \\ for smooth dynamical systems.\\ Preliminary report.}

\author[M.\ Foreman]{Matthew Foreman}

\address{Department of Mathematics, University of California, Irvine, CA~92697, USA}

\author[A.\ Gorodetski]{Anton Gorodetski}

\address{Department of Mathematics, University of California, Irvine, CA~92697, USA}

\thanks{\noindent M.\ F.\ was supported in part by NSF grant DMS-2100367, \\
 A.\ G.\ was supported in part by NSF grant DMS--1855541 }

\begin{document}
\maketitle
\date{}
 \tableofcontents

\noindent{\bf Author's Note} The main results in this paper, Theorems \ref{dim at least five} and \ref{dim two}, were made obsolete by joint results of the authors in March 2022.  The improved results showed the unclassifiability of diffeomorphisms in every dimension.  The results in this paper were proved in the fall of 2019 and winter of 2020.  The authors post this manuscript on arXiv because the proofs in this manuscript illustrate fundamentally different techniques and the manuscript contains expository information that may be useful to researchers not familiar with Borel reducibility.
\bigskip

\section{Introduction}

In his ICM talk in Stockholm in 1962, Stephen Smale proposed a program of studying the topological conjugacy
classes of diffeomorphisms of a given manifold, \cite{S63ICM}. The program was motivated by the related problem on qualitative behaviour of solutions of differential equations on a given manifold.
The roots of the problems came from  Poincare \cite{P} and Birkhoff \cite{B} who studied diffeomorphisms as cross sections of smooth flows--differentiable $\mathbb R$-actions. Smale pointed out that the topological conjugacy class of a diffeomorphism preserved all of its \emph{qualitative behavior}. This includes properties such as having attractors, repelling points, being minimal, or transitive, or topologically mixing, or uniquely ergodic, and so forth. Quoting \cite{S}:

	\begin{quotation}
	``\dots the same phenomena and problems of the qualitative theory of ordinary differential equations are present in their simplest form in the diffeomorphism problem.  Having first found theorems in the diffeomorphism case, it is usually a secondary task to translate the results back into the differential equations framework."
	\end{quotation}

The problem is stated precisely as follows. Fix a compact manifold $M$ and $f, g\in \mbox{Diff}(M)$.  Put $f\sim_{top} g$ if and only if there is a homeomorphism $h:M\to M$ such that $f\circ h=h\circ g$. This equivalence relation is referred to as \emph{topological conjugacy}. Smale suggested classifying  the equivalence classes by ``numerical and algebraic invariants." 

An obvious problem with this definition is that the equivalence relation does not arise as the orbit equivalence relation of a group action. On the surface then, a better notion might be that $f$ is equivalent to $g$ if and only if they are conjugate by a \emph{diffeomorphism} $h$. However Smale rejected this as being too fine an equivalence relation, in part because ``the eigenvalues of the derivative at a fixed point are differentiable conjugacy invariants", and suggested to consider topological conjugacy as the most relevant equivalence relation for smooth dynamical systems.

\medskip
\noindent{\bf Author's Update} In the winter of 2022, P. Kunde showed that the equivalence relation $E_0$ is continuously reducible to the equivalence relation of \emph{conjugacy by diffeomorphisms} on diffeomorphisms of the unit circle with Liouvillean rotation numbers. Thus even in one dimension, that equivalence relation cannot be understood by complete numerical invariants. (See below \hyperlink{SE0}{why $E_0$ is relevant.})

\subsection{Anti-classification: Main results} The main result of this paper is that \emph{Smale's Program} is hopeless in the rigorous sense that there is no \emph{Borel} classification of topological conjugacy, even in the very regular case of $C^\infty$-diffeomorphisms. More precisely we show the following\footnote{After this text was written, we realized that using a different construction one can show that Theorem \ref{dim at least five} in fact holds for any closed manifold (even for the circle). We intend to provide that stronger result in the final version of the paper.}:

\begin{theorem}\label{dim at least five}
Let $M$ be a manifold of dimension at least 5. Consider the space $\mbox{Diff}(M)$ of $C^\infty$-diffeomorphisms of $M$ with the $C^\infty$-topology. Then
$\sim_{top}\subseteq \diff(M)\times \diff(M)$ is not a Borel set in the product topology.

\end{theorem}
For dimension 2 we show that complete numerical invariants don't exist\footnote{In Theorem \ref{dim two} one can replace topological conjugacy by a smooth ($C^r$ for any given $r=0,1, \ldots, \infty$) change of coordinates.}.

\begin{theorem}\label{dim two}  Let $M$ a manifold of dimension at least two. Then there is no Borel function $f$ from $\diff(M)$ to $Y$, where $Y$ is a Polish Space such that $S\sim_{top}T$ if and only if $f(s)=f(t)$.  In particular there is no Borel function $f:\diff(M)\to \mathbb R^n$ that gives complete numerical invariants for the equivalence relation $\sim_{top}$ on  $\diff(M)$.
\end{theorem}

\medskip

It remains to explain why these results show that classification is impossible.

In Section \ref{s.2} we prove Theorem \ref{dim two}. In Section \ref{s.5} we prove Theorem \ref{dim at least five}. In Section \ref{s.questions} we formulate some open questions and conjectures related to our results.

\subsection{Impossibility Results and Generalized Church's Thesis} Many of the most famous results in mathematics are \emph{impossibility results}: You can't square the circle using a ruler and compass, you can't solve the general quintic polynomial using radicals and so forth.  What these results have in common is a precise description of what it would mean to solve the problem.

In the 20th century, impossibility results took a different turn.  The tools of G\"odel and Turing made it possible to show results of the form:
	\begin{quotation}
	\noindent It is impossible to solve the problem P using purely finitary methods.
	\end{quotation}
But this demands an explanation of what the definition of \emph{finitary methods} means.  The rigorous interpretation of a problem being solvable with finitary methods is that it can be solved by a  \emph{recursive algorithm}:  there is a Turing machine that can check each instance and correctly give the answer.\footnote{More formally, there is a Turing machine with a finite set of instructions that, on input of the appropriate data always halts in finite time and outputs the correct answer. There are no time or space limitations, such as \emph{polynomial time} on the computation.}

Some famous examples of impossibility results include:
	\begin{description}
	\item[Novikov:] The word problem is not recursive: Given a finite presentation of a group it is impossible to tell whether an arbitrary given word is trivial in the group.
	\item[Davis, Matiysevich, Putnam, Robinson:] Hilbert's 10th problem is undecidable:  it is impossible to tell whether an arbitrary given diophantine polynomial has a solution with strictly finitary methods.
	\end{description}
The meta-assertion that all notions of finitary computation are equivalent is referred to as \emph{Church's Thesis}.
\medskip

\bfni{A Generalized Church's Thesis.} Theorems \ref{dim at least five} and \ref{dim two} are from a genre of  even stronger impossibility results.   They can be interpreted as saying:

	\begin{quotation}
	\noindent It is impossible to solve a problem P using inherently countable techniques.
	\end{quotation}
For this to make sense the problem has to be posed in the context of a Polish space-- a
topological space $X$ whose topology is compatible with a complete separable metric.

The notion of a Polish space captures the idea that a countable amount of information can be used to completely specify the problem P.  The countable dense set $D$ and a compatible complete metric $d(x,y)$ on $D\times D$ constitutes a countable amount of information. Since $D$ is dense
every point of $X$ is a limit of elements of a fixed countable set.

Having fixed $D$ and $d$, basic open intervals, consisting of rational neighborhoods of this dense set are then viewed as countable information.  Having basic open intervals, countable Boolean operations (union, intersection and complement) then exhaust what one can do inherently countably. The result is the collection of Borel subsets of $X$.

As summarized above, a Church's Thesis can be viewed as stating that a problem P  is recursive if and only if it can be solved using inherently finite techniques in any sense of \emph{inherently finite}. The generalized Church's thesis we are suggesting is that a problem can be solved with \emph{inherently countable techniques} just in case it can be viewed as  \emph{Borel in a Polish Space}.

Is this a reasonable definition of inherently countable techniques? It may be too liberal--perhaps the metric on the Polish Space itself uses the uncountable axiom of choice to construct. But we will be showing things are \emph{not Borel} so our results are even stronger if the notion is too liberal.

But is it too restrictive? Are there methods that should be considered inherently countable that are not Borel?  Any construction that involves taking countable limits preserves the property of being Borel.  Anything that can reasonably be called a \emph{construction} that takes countably many steps is Borel.

Moreover Borel sets are preserved under pre-images of Borel functions. Hence any computation that
results in some sort of complete computational invariant for an equivalence relation implies that that
equivalence relation is Borel.  As a non-trivial example, the Spectral Theorem for Normal Operators
gives complete invariants for the equivalence relation of Unitary Equivalence, showing that
equivalence relation of Unitary Equivalence of Normal Operators is a Borel equivalence relation.

To clarify the idea, we give  example of some invariants commonly attached to a diffeomorphism.

We will denote by $\text{Fix}(f)$ the set of fixed points of the map $f:M\to M$, i.e. $$\text{Fix}(f)=\{x\in M\ |\ f(x)=x\},$$ and by $\text{Per}_n(f)$ the set of periodic points of period $n$, i.e. $$\text{Per}_n(f)=\{x\in M\ |\ f^n(x)=x\}.$$

\begin{prop}\label{p.pp} Let $M$ be a compact manifold.
For any $r=0, 1, 2, \ldots, \infty$ the following statements hold:
	\begin{enumerate}
	\item The set $\mathcal{U}_0=\{f\in \text{Diff}^{\,\, r}(M)\ |\ \text{Fix}\,(f)=\emptyset\}$ is open; 
	
\vspace{4pt}

	\item For each $m\in \mathbb{N}$ the set $\mathcal{U}_m=\{f\in \text{Diff}^{\,\, r}(M)\ |\ \#\text{Fix}\,(f)=m\}$ is Borel; %
	
\vspace{4pt}

	\item For each $m, n\in \mathbb{N}$ the set $\mathcal{U}_m^n=\{f\in \text{Diff}^{\,\, r}(M)\ |\ \#\text{Per}_n(f)=m\}$ is Borel;  
\vspace{4pt}

    \item  For any sequence of integers $A=\{a_n\}$, $a_n\ge 0$, the set $$\mathcal{U}_A=\{f\in \text{Diff}^{\,\, r}(M)\ |\ \#\text{Per}_n(f)=a_n, \ n=1, 2, \ldots \}$$ is Borel;

    \vspace{4pt}

    \item The set $$\{(A, f)\ |\  \#\text{Per}_n(f)=a_n, \ n=1, 2, \ldots,\ \text{where}\ A=\{a_n\}  \}$$ in the space $\mathbb{N}^{\mathbb{N}}\times \text{Diff}^{\,\, r}(M)$ is Borel;

    \vspace{4pt}

    \item The function $G:\text{Diff}^{\,\, r}(M) \to \mathbb{R}\cup\{\pm\infty\}$ defined by
    $$
    G(f)=\limsup_{n\to \infty}\frac{1}{n}\log \#\text{Per}_n(f)
    $$
    (the exponential rate of growth of the number of periodic points) is Borel. In particular, the set of Artin-Mazur diffeomorphisms (defined by $\{f:G(f)<\infty\}$) is Borel.
	\end{enumerate}
\end{prop}
\begin{remark}
Certainly, one could add many other Borel sets and Borel functions that can be expressed in turns of the number of periodic point to the statement of Proposition \ref{p.pp}. We specifically mentioned the rate of growth of the number of periodic points, since it is the function that was heavily studied in the theory of dynamical systems (e.g. see \cite{AM, HK, K1, K2, KS}).
\end{remark}
\begin{proof}[Proof of Proposition \ref{p.pp}]
It is clear that the set $\mathcal{U}_0$ is open.

To see that $\mathcal{U}_1$ is Borel, let us choose a countable base of topology on $M$, $\{V_k\}_{k\in \mathbb{N}}$, such that
$$
\text{diam}\, V_1\ge \text{diam}\, V_2\ge \text{diam}\, V_3\ge \ldots, \ \ \text{and}\ \ \text{diam}\, V_k\to 0\ \ \text{as}\ k\to \infty.
$$
The set
$$
\mathcal{V}_k=\{f\in \text{\it Diff}^{\,\, r}(M)\ |\ f\ \ \text{has no fixed points outside of}\  V_k\}
$$
is open, and hence the set
$$
\mathcal{U}_1=\left(\bigcap_{K\ge 1} \bigcup_{k=K}^\infty \mathcal{V}_k\right)\backslash \mathcal{U}_0
$$
is Borel.

To see that $\mathcal{U}_2$ is Borel, notice that each of the sets
$$
\mathcal{V}_{k_1, k_2}=\left\{f\in \text{\it Diff}^{\,\, r}(M)\ |\ f\ \ \text{has no fixed points outside of}\  V_{k_1}\cup V_{k_2}\right\}.
$$
is open, and hence
$$
\mathcal{U}_2=\left(\bigcap_{K\ge 1} \bigcup_{k_1, k_2\ge K} \mathcal{V}_{k_1, k_2}\right)\backslash \left(\mathcal{U}_0 \cup \mathcal{U}_1\right)
$$
is Borel. Similarly one can show that $\mathcal{U}_m$ is Borel for each $m\in \mathbb{N}$.

Since for each given $n\ge 2$ the map $f\mapsto f^n$ is continuous, the set $\{f^n\ \text{has no fixed points outside of}\ V_k\}$ is open, and the arguments similar to those above prove that the sets $\mathcal{U}_m^n$ are Borel. Also, one has $\mathcal{U}_A=\bigcap_{n}\mathcal{U}_{a_n}^n$, hence $\mathcal{U}_A$ is Borel.

To show that the set $$\{(A, f)\ |\  \#\text{Per}_n(f)=a_n, \ n=1, 2, \ldots,\ \text{where}\ A=\{a_n\}  \}$$ is Borel, notice that it can be represented as $\bigcup_{n\in \mathbb{N}}X_n$, where
$$
X_n=\bigcup_{m\in \mathbb{N}}\left\{(A, f)\ |\ \text{$n$-th element of $A$ is $m$, and }\, f\in U_m^n \right\}
$$
is a Borel set.
Finally,
$$
\{f\in \text{Diff}^{\,\, r}(M)\ |\ G(f)<C\}=\bigcup_{k\in \mathbb{N}}\bigcup_{M\in \mathbb{N}}\bigcap_{n\ge M}\bigcup_{m<\exp(nC-\frac{n}{k})}U_m^n,
$$
hence $G:\text{Diff}^{\,\, r}(M)\to  \mathbb{R}\cup\{\pm\infty\}$ is a Borel function.
\end{proof}

In contrast, Theorem \ref{dim at least five} says that the equivalence relation of topological conjugacy is \emph{not} Borel for manifolds of dimension 5 or greater, and Theorem \ref{dim two} claims that there is no way to provide complete numerical invariants that would be Borel functions of a diffeomorphism and provided classification up to a topological conjugacy for manifolds of dimension 2 or greater. Thus it precludes any reasonable complete invariants for the qualitative behavior of these diffeomorphisms.

Certainly, if we restrict ourselves to some smaller class of diffeomorphisms, or consider phase space of small dimension (say, the circle), then it is reasonable to expect that a classification up to topological conjugacy can exist. In Section \ref{ss.examples} below we give some examples of such results.

\subsection{Examples}\label{ss.examples}

To put our results into context, here we provide some examples of successful classification of some classes of smooth dynamical systems, as well as some known results on absence of such classification (mostly regarding topological and measure preserving dynamics). We would like to emphasize that both lists are far from being even remotely complete, and are provided as an illustration rather than a survey of all relevant results. Another  list of problems will appear in the forthcoming paper by Foreman, \emph{The complexity of the Structure and Classification of Dynamical Systems}.

\medskip

\noindent{\bf Examples of successful classifications in smooth dynamics:}

\vspace{5pt}

$\bullet$ Real analytic diffeomorphisms of an interval: Any two homeomorphisms of an interval $[0,1]$ that do not have any fixed points except $0$ and $1$ are topologically conjugate, as can be shown using a method of fundamental domain. Using this fact, one can show that there exist complete numerical invariants for a topological conjugacy on the space of real analytic diffeomorphisms of the unit interval. 

\vspace{5pt}

$\bullet$ Classification of real analytic diffeomorphisms of the circle up to a topological conjugacy: The rotation number is an invariant of (orientation preserving) topological conjugacy. Due to Denjoy's Theorem \cite{D}, any $C^2$ diffeomorphism with irrational rotation number is topologically conjugate to a rotation of the circle. 
For real analytic diffeomorphisms with rational rotation number the number, type (stable/unstable/semistable), and relative order of the periodic points form a complete invariant of topological conjugacy.
See \cite{KH} for more details on rotation numbers and Poincare's classification.

\vspace{5pt}

$\bullet$  Franks--Newhouse Theorem on Anosov (uniformly hyperbolic) diffeomorphisms: If $f$ is an Anosov diffeomorphism
of co-dimension one (i.e. either stable or unstable direction has dimension one) on a compact Riemannian manifold $M$, then $f$ is topologically
conjugate to a hyperbolic toral automorphism, see \cite{Fr} and \cite{N_Anosov}, and also \cite{Hi} for an alternative proof.

\vspace{5pt}

$\bullet$  Anosov maps on tori: Any Anosov  diffeomorphism of a torus is topologically conjugate to a linear automorphism of a torus, and two linear automorphisms of a torus are topologically conjugate if and only if the corresponding matrices are conjugate by elements of $GL_2(\mathbb Z)$, see \cite{Fr69}, and also \cite[Section 18.6]{KH}. The linear automorphism of the torus associated with an Anosov Diffeomorphism can be computed with a Borel function. 
This provides an explicit classification of Anosov maps of a torus.

\vspace{5pt}

$\bullet$  Structurally stable diffeomorphisms: Fix a {compact} manifold $M$.  A $C^r$ diffeomorphism $f:M\to M$ is \emph{structurally stable} if there is a $C^1$-open set $O$ containing $f$ such that every $g\in O$ is topologically conjugate to $f$.
The collection of structurally stable diffeomorphisms is open and there can only be countably many
$\sim_{top}$ equivalence classes of structurally stable diffeomorphisms. 
It follows immediately that they are classifiable in the sense that there is a (locally constant) function $R$ from the $C^r$-structurally stable diffeomorphisms to the natural numbers so that two structurally stable diffeomorphisms $f, g$ are conjugate if and only if $R(f)=R(g)$. Extending this function so that it takes value $-1$ on the complement  of the open set of structurally stable diffeomorphisms gives a Borel classification of the structurally stable diffeomorphisms by complete numerical invariants. 

However, it is not a classification of all diffeomorphisms that are topologically conjugate to a structurally stable diffeomorphism. This is an open problem, as of this writing.

\vspace{5pt}

$\bullet$ An important case of structurally stable diffeomorphisms that in lower dimensions can be classified by very explicit invariants is the collection of Morse-Smale diffeomorphisms.  To define Morse-Smale diffeomorphisms, we need the following definition:
\bigskip 

\noindent{\bf Definition} A point $x$ is a wandering point of the map $f$ if it has a neighborhood $U=U(x)$ such that $f^n(U)\cap U=\emptyset$ for all $n\ge 1$. The non-wandering set of $f$ is the set of all points that are not wandering.
\medskip

\begin{definition}\label{d.MS}
A diffeomorphism $f$ is \emph{Morse-Smale} if
\begin{enumerate}
\item $f$ has only finite number of periodic orbits {and these} form the non-wandering set\ of $f$,
\item all of the periodic orbits are hyperbolic, and
\item the stable and unstable manifolds of any two periodic points intersect transversally.
\end{enumerate}
\end{definition}

\begin{remark}\label{r.MS}
A well-understood method of constructing some of the Morse-Smale diffeomorphisms uses gradient flows.
Let $\{T_t:t\in \mathbb R\}$ be a
flow on $M$ such that
	\begin{enumerate}
	\item There is a function $L:M\to \mathbb R$ such that $L$ has a finite number of critical points
	$x_1, \dots x_n$, all of the critical points are non-degenerate and are such that for $i\ne j$, $L(x_i)\ne L(x_j)$.
		\item $T_t$ is the gradient flow for $L$.
	\end{enumerate}
Then the time-1 shift $f=T_1$ for this flow is a Morse-Smale diffeomorphism.
\end{remark}

Because Morse-Smale flows are structurally stable, the previous discussion shows that they have an abstract classification by numerical invariants.  However, more is known: there are classifications that give conceptual information about their properties.

Morse-Smale flows on surfaces were classified in \cite{Pei}. A different set of invariants was suggested in \cite{OS}. A complete topological classification of Morse-Smale diffeomorphisms on compact orientable surfaces was obtained in \cite{BL}, see also \cite{BG} and \cite{G93}. 
Topological classification of Morse-Smale diffeomorphisms on 3-manifolds was given in \cite{BGP}; a complete topological invariant in this case is shown to be an equivalence class of a ``scheme'', which contains an
information on a periodic data and a topology of embedding of two-dimensional invariant
manifolds of the saddle periodic points of a diffeomorphism into the ambient manifold. For a detailed historical background of the topological classification of Morse-Smale systems see Section 1.4 in \cite{BGP}.

In Remark \ref{on the boundary} we discuss how to give a family of examples of unclassifiable diffeomorphisms that live on the boundary of the Morse-Smale diffeomorphisms.

\medskip

\noindent{\bf Examples of anti-classification:}
 
\vspace{5pt}

  An early dramatic success of anticlassification results was the theorem due independently to Kaufman and
Solovay  (building on work of Debs, St. Raymond, Kechris and Louveau) who showed in 1983-84,
 that the collection of closed \emph{sets of uniqueness} for trigonometric functions is not a
Borel set. (A set $E\subseteq [0,1]$ is a \emph{set of uniqueness} if whenever $\sum c_ne^{2\pi inx}$
on $[0,1]\setminus E$ the series $\sum c_ne^{2\pi inx}$ is identically 0.)  Hence determining whether
the complement of a  given closed set determines the values of a trigonometric series (a problem
posed by Kronecker, and studied by Cantor and many others) is simply not possible using anything
resembling even a countable transfinite  computation.

 The first applications of the Borel/non-Borel distinction in dynamical systems were the results in \cite{BF1} and \cite{BF2} which showed the the collections of minimal topologically  distal transformations and, respectively,  generalized discrete spectrum transformations are not Borel.\footnote{For a compact separable metric space $X$, a homeomorphism of $X$ is \emph{distal} if and only if for all  $x\ne y$ there is an
$\epsilon>0$ for all $n, d(T^nx,T^ny)>\epsilon$. The definition dates to Hilbert.}  It has been used more recently in \cite{part1}, \cite{part2} and \cite{part3} to show that ergodic diffeomorphisms of compact manifolds are not classifiable.

\subsection{Borel sets, analytic sets and reductions}
 We now give some basic definitions.
Fix a Polish space $(X,\tau)$.  A set is \emph{Borel} if and only if it is in the smallest $\sigma$-algebra containing the open set $\tau$.  If $X, Y$ are Polish spaces and $f:X\to Y$ is continuous, then for all Borel sets $B$ the $f$-image of $B$, $f[B]$ is called an \emph{analytic} set.

The simplest way of generating an analytic set is to start with a Borel subset $B\subseteq X\times Y$ and taking its projection to the $Y$-axis.  Indeed all analytic subsets of $Y$ can be built this way.\footnote{For this and other information about Borel and analytic sets see \cite{naive DST}, \cite{Marker website} or \cite{Kechris}.}
The claim that all projections of Borel sets is Borel is a famous mistake of Lebesgue that
was rectified by examples of Suslin.

Results such as Theorems \ref{dim at least five} and \ref{dim two} are shown using \emph{reductions},  a technique due to Friedman and Stanley  and developed independently by Harrington, Kechris and Louveau.

The idea starts with the clich\'e  that:
\begin{quotation} To solve $A$ you reduce it to a
problem $B$ which you already know how to solve.
\end{quotation}
Taking the contrapositive:
\begin{quotation} To show that solving $B$ is impossible, you start with a
\emph{known} impossible problem $A$ and reduce it to $B$.
\end{quotation}

\noindent Formally:
\begin{definition}\notag
Let $X$ and $Y$ be Polish spaces, $A\subseteq X$, $B\subseteq Y$. Then $A$ is Borel
 reducible to $B$ if and only if there is a Borel
   function $f:X\to Y$ such that for $x\in X$:
\[x\in A \mbox { if and only if } f(x)\in B.\]
\end{definition}
Thus if $A$ is not Borel, $B$ cannot be either, since the inverse image of a Borel set by a Borel function is not Borel. The function $f$ is a \emph{Borel reduction}.

Define $A\preceq_\mcb B$ if $A$ is Borel reducible to $B$. Then $\preceq_\mcb$ is transitive since one can compose Borel reductions.  Defining the equivalence relation $A\sim_\mcb B$ if $A\preceq_\mcb B$ and $B\preceq_\mcb A$ we see that $\preceq_\mcb$ induces a partial ordering of the $\sim_\mcb$ equivalence classes.

 The heuristic above interprets $A\preceq_\mcb B$ as saying that $B$ is at least as complicated as $A$ (with respect to countably feasible computations) and $A\sim_\mcb B$ as saying that they have the same complexity.  Among analytic sets, there is a $\preceq_\mcb$-maximal equivalence class, the \emph{complete} analytic sets. Since the Borel sets are $\preceq_\mcb$-downwards closed, and there are non-Borel analytic sets, any complete analytic set cannot be Borel.
 \medskip

\bfni{Two dimensional reductions and equivalence relations}

\vspace{5pt}

Let $E$ and $F$ be equivalence relations on Polish spaces $X$ and $Y$ respectively. We define
$E\preceq_\mcb F$ if and only if there is a Borel function $f:X\to Y$ such that for all $x_1, x_2\in X$,
	\begin{quotation} \noindent
	$x_1 E x_2$ if and only if $f(x_1) F f(x_2)$.
	\end{quotation}
We note that this differs from the definition  above because it relates pairs of elements of $X$ and
$Y$ that are connected by a unary function--hence a \emph{two dimensional reduction}.

For the notion of reduction to be useful, one must have complicated examples. Indeed \cite{classDS}
describes a whole collection of benchmark equivalence relations that can be used to calibrate the complexity of a given equivalence relation.
\bigskip

\bfni{All uncoutable Polish Spaces are Borel Isomorphic.} It is a classical fact that all uncountable Polish spaces are Borel isomorphic. (See, for example \cite{naive DST}.) Hence the particular spaces being used for  reductions are chosen so the the mathematical tools appropriate for the context are available.  The strongest classifications are Borel reductions to the identity  equivalence relation on a Polish space.  Because every Polish space can be injected into $[0,1]$ (or the irrational numbers in $[0,1]$), we call these \emph{numerical invariants}.  In particular to show that there is \emph{no} reduction to the identity relation on any Polish space it suffices to show there is no reduction to the identity equivalence relation on $[0,1]$.

\subsection{The space of countable graphs and Theorem \ref{dim at least five}}
\hypertarget{graphs}{One prominent example consists of the space}
$\mbox{\emph{Graphs}}\subseteq \{0, 1\}^{\nn\times \nn}$ consisting of those
$f:\nn\times \nn\to \{0,1\}$ such that for all $(n,m), f(n,m)=f(m,n)$.  If $G=\la \nn, E\ra$ is a graph with vertices
$\nn$ and edges $E$, then we can associate the function $f_G:\nn\times\nn\to \{0,1\}$ by putting $f_G(m,n)=1$
just in case $m$ and $n$ are connected by an edge $E$.  This process can be reversed:  any symmetric 
$f:\nn\times \nn\to \{0, 1\}$ defines a graph $G_f$ by putting an edge between $m$ and $n$
if $f(m,n)=1$. Hence we  can identify graphs with vertex set $\nn$ with elements of the subspace \emph{Graphs}$\ \subseteq \{0,1\}^{\nn\times \nn}$. We will say that the function $f_G$ \emph{codes} the graph $G$.

Since the space $\{0,1\}^{\nn\times \nn}$ can be viewed as a countably infinite product of the two point
space $\{0,1\}$, it can be endowed with the product topology. With this topology,  $\{0,1\}^{\nn\times \nn}$ is
homeomorphic to the Cantor set. The collection \emph{Graphs} $\subseteq \{0,1\}^{\nn\times \nn}$ is a
closed subset, and hence the induced topology makes it a Polish space.
\smallskip

Let $S_\infty$ be the space of (arbitrary) permutations of $\nn$. Then $S_\infty$ acts on \emph{Graphs} by setting, for $\phi\in S_\infty$ and $f_G\in $ \gfs,
	\begin{equation}
	\label{action}
	(\phi f_G)(m,n)=1 \emph{ if and only if } f_G(\phi^{-1}m,\phi^{-1}n)=1.
	\end{equation}
Then $\phi f_G$ is a graph with vertex set $\nn$.

It is routine to check that for graphs $G, H$ with vertices $\nn$, $G$ and $H$ are isomorphic just in case $f_G$ and $f_H$ are in the same $S_\infty$ orbit.\footnote{In the jargon of Hjorth's influential book, \emph{Isomorphism of Countable Graphs} is an example of an equivalence relation ``classifiable by countable models."  Indeed the graphs themselves are the models.}
\medskip 

Here are two standard results:

\begin{theorem*}
Consider the isomorphism relation $\sim_{iso}$ of \emph{Graphs}. Then:
	\begin{enumerate}
	\item As a subset of the Polish space \emph{Graphs} $\times$ \emph{Graphs}, the relation
	 $\sim_{iso}$ is complete analytic.
	\item If $X$ is a Polish space and $E\subset X\times X$ is the orbit equivalence relation arising from an $S_\infty$-action then $E\preceq_\mcb \sim_{iso}$.
	\end{enumerate}
\end{theorem*}

We can now restate the main theorem of this paper:

\begin{theorem}\label{main result} Let $M$ be a manifold of dimension at least 5. Then
the isomorphism relation 
on \emph{Graphs} is Borel reducible to the equivalence relation $\sim_{top}$ on $\diff(M)$.
\end{theorem}
 Indeed, Theorem \ref{dim at least five} follows immediately from this result.

\medskip

\subsection{The relation $E_0$ and Theorem \ref{dim two}} 

In this section we discuss assigning complete numerical invariants to an equivalence relation. Given any two uncountable Polish spaces $X,Y$ there is a Borel bijection between $X$ and $Y$. Hence all uncountable Polish spaces are Borel isomorphic to the irrational numbers in $[0, 1]$. For this reason the phrase \emph{complete numerical invariant} is used to refer to invariants $x$ belonging to any fixed Polish space $X$.
\medskip

\hypertarget{SE0}{
\bfni{The  first anti-classification result in dynamical systems}}  The Kolmogorov-Sinai entropy is a function that computes a positive real number in a Borel way. Ornstein \cite{O} showed that for Bernoulli shifts, the entropy is a complete invariant: two Bernoulli shifts that have the same entropy are isomorphic by a measure preserving transformation. 
After Ornstein's classification there was optimism that the methods could be extended to larger collections.

The natural first step was to $\mck$-automorphisms, those ergodic measure preserving
transformations where ``the present is asymptotically independent of the past."
Ornstein and Shields produced an uncountable family of non-isomorphic $\mck$-automorphisms of
any given entropy--thus ruling out entropy as a potential complete invariant.

Feldman (\cite{JF}) gave the first general anti-classification result that showed that there is \emph{no} numerical invariant that completely classifies the $\mck$-automorphisms up to measure isomorphism.
We describe his result in contemporary language.

 \hypertarget{defE0}{The benchmark equivalence relation he considers is called $E_0$.} It is an equivalence relation on the Cantor set. Viewing the Cantor set as $\{0,1\}^\nn$,  two elements $f, g$ are $E_0$ equivalent if and only if there is an $N$ for all $n\ge N, f(n)=g(n)$.\footnote{In general we use the notation $X^Y$ for the space of functions from $Y$ to $X$, so $\{0, 1\}^\nn$ consists of all functions from $\nn$ to $\{0, 1\}$; all countable sequences of $0$'s and $1$'s.} 

The significance of $E_0$ is captured by the following theorem:

\begin{theorem*}(Harrington, Kechris, Louveau)\label{HKL}
Let $E$ be an analytic equivalence relation on a Polish space $X$.  Then exactly one of the following holds:
	\begin{description}
	\item[$E$ is smooth] There is a Borel function $f:X\to Y$ where $Y$ is a Polish space that reduces $E$ to the identity equivalence relation on $Y$.
	\item[$E_0$ reduced to $E$] The equivalence relation $E_0$ is reducible to $E$.
	\end{description}
\end{theorem*}

The first case means that $f$ attaches  complete numerical invariants (elements of $Y$) to $E$. The theorem then says that either there is a Borel function that computes complete numerical invariants for $E$ or $E_0$ is reducible to $E$.

Feldman's theorem can be restated as saying:
\begin{theorem*}
There is a continuous map $f:2^\nn\to \{\mck$-automorphisms$\}$ that reduces $E_0$ to measure isomorphism.
\end{theorem*}

For the reader's benefit we sketch the argument for the incompatibility of the two alternatives described in the Harrington, Kechris, Louveau theorem. This is the portion of their theorem we use.

\begin{proposition}\label{p.reduction}
Suppose that $E$ is an equivalence relation on a Polish space $X$ and $E_0$ is Borel reducible to $E$.  Then there is no Borel function $g:X\to Y$ with $Y$ a Polish space such that for all $x_1, x_2\in X$,

	\[x_1 E x_2 \mbox{ if and only if } g(x_1)=g(x_2).\]
\end{proposition}
\pf By the earlier remarks after the definition of two dimensional reductions, we can assume that $Y=[0,1]$.

If $f:\{0,1\}^\nn\to X$ is the reduction of $E_0$ to $E$ then $g\circ f:\{0,1\}^\nn\to [0,1]$ and for all $\vec{x}, \vec{y}\in \{0,1\}^\nn$, we have $\vec{x}$ eventually agrees with $\vec{y}$ if and only if $(g\circ f)(\vec{x})=(g\circ f)(\vec{y})$. So without loss of generality we can take $E=E_0$ and $f$ is a reduction to the identity equivalence relation on $[0,1]$.

Since $f$ is a Borel map, for each subinterval $I$ of $[0,1]$, $f^{-1}(I)$ has the property of Baire. Fix
an $n$ and look at decimal expansions of elements of $[0,1]$ of length $n$.  These partition $[0,1]$ into $10^n$ many sets. For each $\vec{a}=a_0a_1\dots a_n$ with $0\le a_i\le 9$, let $A_{\vec{a}}$ be the set of $x$ in $[0,1]$ such that the first $n$ digits of $x$ are $a_0a_1\dots a_n$.
 By the Baire Category Theorem,  there must be some $\vec{a}$  such that $f^{-1}(A_{\vec{a}})$  is not first category.\footnote{In Descriptive Set Theory the terms ``meager'' and ``comeager'' are usually used, while in Dynamical Systems their synonyms, ``first category'' and ``residual,'' are mostly preferred.}

A standard result proved using the  Banach-Mazur game \cite{wikipediaBM} says that there must be a basic open interval $[s]$  in the product topology on $\{0,1\}^\nn$ given by a finite sequence of $0's$ and $1's$ such that $f^{-1}(A_{\vec{a}})$ is residual in $[s]$. Since equivalence by $E_0$  just depends on tails of elements $\vec{x}\in \{0,1\}^\nn$, $f^{-1}(A_{\vec{a}})$ must be residual in $\{0,1\}^\nn$.

Hence, again by the Baire Category theorem, $\vec{a}$ must be unique. Denote the sequence of length $n$ by $\vec{a}_n$. Since  $f^{-1}(A_{\vec{a}_n})$ and $f^{-1}(A_{\vec{a}_{n+1}})$ are both residual we must have $\vec{a}_{n+1}$ extending $\vec{a_n}$.

This defines a collection $\la \vec{a}_n:n\in\nn\ra$ of  finite subsequences of $\{0, 1, \dots 9\}$ that extend each other.  If $a\in [0,1]$  has decimal expansion given by $\la \vec{a}_n:n\in\nn\ra$  then, again by the Baire Category Theorem,  $f^{-1}(a)$ is residual.

Let $B=f^{-1}(a)\subseteq \{0,1\}^\nn$.  Since $B$ is residual, it contains a perfect set--in particular it contains $x, y$ that disagree infinitely often. But then $x\not\!\! E_0 y$ and $f(x)=f(y)$, a contradiction.\qed

We can now state the content of Theorem \ref{dim two} more precisely:

\begin{theorem}\label{E0 dim two}
Let $M$ be a manifold of dimension at least two and $\sim_{top}$ the relation of topological equivalence on $\diff(M)$.   Then there is a Borel reduction of $E_0$ to
$\diff/\sim_{top}$.
\end{theorem}

\section{Two dimensional diffeomorphisms}\label{s.2}

Here we prove Theorem \ref{E0 dim two}. By Proposition \ref{p.reduction} that will imply Theorem \ref{dim two}.

\vspace{5pt}

Let $$
\Sigma=\{0,1\}^{\mathbb{N}}=\{\ \bar \omega=\omega_1\omega_2\ldots, \ \omega_i\dots \in \{0,1\}\ \}.
$$

The standard topology makes  $\Sigma$ homeomorphic to the Cantor set.  This topology is induced by the metric: 
$$d(\bar\omega, \bar\omega')=\left\{
                                                                                            \begin{array}{ll}
                                                                                              0, & \hbox{if $\bar\omega=\bar \omega'$;} \\
                                                                                              2^{-k}, & \hbox{if $\omega_i=\omega'_i$ for $i<k$, and $\omega_k\ne \omega_k'$.}
                                                                                            \end{array}
                                                                                          \right.
$$

Recall the \hyperlink{defE0}{definition of $E_0$}:
$$
\bar\omega\sim_{E_0} \bar\omega'\ \ \ \text{if and only if}\ \ \ \omega_i=\omega_i'\ \ \ \text{for all sufficiently large} \ \ i\in \mathbb{N}.
$$
In other words, $\bar \omega\sim_{E_0} \bar \omega'$ iff $\bar \omega$ and $\bar \omega'$ belong to the inverse orbit of the same point under the topological Bernoulli shift.
\begin{theorem}\label{l.1}
There exists a continuous map $\mathbf{F}:\Sigma\to \text{Diff}^\infty(\mathbb{R}^2)$ such that $\mathbf{F}(\bar\omega)$ is topologically conjugate to $\mathbf{F}(\bar \omega')$ if and only if $\bar\omega\sim_{E_0} \bar\omega'$.
\end{theorem}

Since the support of each diffeomorphism in the range of $F$ is compact, 
 the proof of Theorem \ref{l.1} for $\text{Diff}^\infty(\mathbb{R}^2)$ works for $\text{Diff}^{\,r}(M)$ for any $r=0, 1, \ldots, \infty$, and any two dimensional surface $M$.  It is easily adapted to any manifold of dimension at least two as discussed in Corollary \ref{E0anydim}.
 \medskip

Before giving a formal proof, let us provide an informal description of the construction.

{First we will choose two points $a, b\in \mathbb{R}^2$, and two sequences of disjoint closed discs $D(c_n^a, r_n)$ and $D(z_n^a, r_n)$   converging to the point $a$ and corresponding discs  $D(c_n^b, r_n)$ and $D(z_n^b, r_n)$ converging to the point $b$. These are shown in figure \ref{basicdiscs}.}

For each $\bar \omega\in \Sigma$, the map $\mathbf{F}(\bar \omega)$ will be supported on the union of these discs, and
the restrictions of $\mathbf{F}(\bar \omega)$ to any two of those discs are going to be different (i.e. not topologically
conjugate). {The restrictions $\mathbf{F}(\bar \omega)|{D(c_n^{a}, r_n)}$ and
$\mathbf{F}(\bar \omega)|{D(z_n^{a,}, r_n)}$ will  tend to the identity much faster than
$c_n^{a}$ or $z^{a}_n$ converge to $a$ and the analogous fact holds for the discs corresponding to $b$. This  ensures
smoothness of the map at the points $a$ and $b$. The restrictions $\mathbf{F}(\bar \omega)$ to the discs
${D(c_n^{a}, r_n)}$ and ${D(c_n^{b}, r_n)}$ will not depend on $\omega$.  They are needed to make sure that if
$\mathbf{F}(\bar \omega)$ and $\mathbf{F}(\bar \omega')$ are topologically conjugate, then the conjugacy fixes the
points $a$ and $b$. }

\begin{figure}[htp] \includegraphics[width=0.8\textwidth]{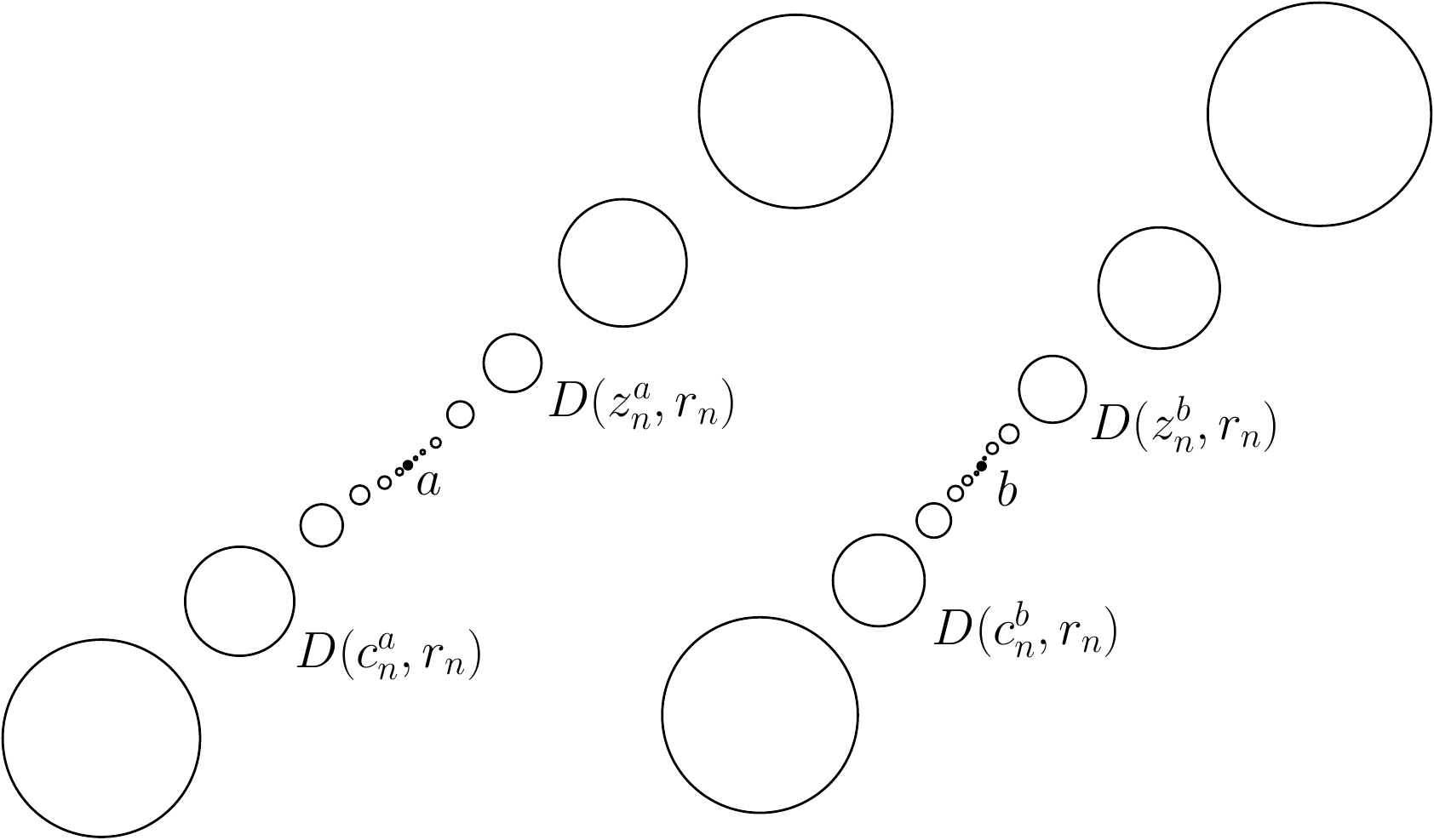}
 \caption{The construction of $\mathbf{F}(\bar\omega)$.}
 \label{basicdiscs}
\end{figure}

Let us now explain how we choose the {restrictions of  $\mathbf{F}(\bar \omega)$ to ${D(z_n^{a}, r_n)}$ and
${D(z_n^{b}, r_n)}$. First fix pairwise different maps $f_{n}, g_{n}$ on discs of radius $r_n$.
 For any sequence $\bar \omega=\omega_1 \omega_2\ldots$ with $\omega_i\in \{0,1\}$:
 \begin{enumerate}
 \item if $\omega_n=0$ we let  $\mathbf{F}(\bar \omega)|{D(z_n^{a}, r_n)}=f_n$ and $\mathbf{F}(\bar \omega)|{D(z_n^{b}, r_n)}=g_n$,
 \item if $\omega_n=1$, then we will swap them:  $\mathbf{F}(\bar \omega)|{D(z_n^{a}, r_n)}=g_n$ and
 $\mathbf{F}(\bar \omega)|{D(z_n^{b}, r_n)}=f_n$.
 \end{enumerate}
 We will show that $\bar \omega\sim_{E_0}\bar \omega'$, if and only if  $\mathbf{F}(\bar \omega)$ is topologically conjugate to $\mathbf{F}(\bar \omega')$.}

Let us now provide the formal proof.

\begin{proof}[Proof of Theorem \ref{l.1}] {For each $\bar\omega\in \Sigma$ we need to define a diffeomorphism $\mathbf{F}(\bar\omega)$ of $\mathbb{R}^2$.
We denote $\mathbf{F}(\bar\omega)$ by $f_{\bar\omega}$.}

{The diffeomorphism  $f_{\bar\omega}$  is the identity on the complement of the union
\[ D(c_n^a, r_n)\cup D(c_n^b, r_n)\cup D(z_n^a, r_n)\cup D(z_n^b, r_n).\]
It remains to define $f_{\bar\omega}$ on the discs themselves.}

 Fix two points $a, b\in \mathbb{R}^2$ such that $\text{dist}\, (a, b)=\frac{1}{2}$. Choose four sequences $\{c_n^a\}$, $\{c_n^b\}$, $\{z_n^a\}$, and $\{z_n^b\}$, $n\ge 2$,  of points on the torus such that for each $n\ge 2$ we have
$$
\text{dist}\, (a, c_n^a)=\text{dist}\, (b, c_n^b)=\frac{1}{n^4},
$$
and
$$
\text{dist}\, (a, z_n^a)=\text{dist}\, (b, z_n^b)=\frac{1}{2n^4}.
$$
Set $r_n=\frac{1}{n^{10}}$, $n\ge 2$.

For any $c\in \mathbb{R}^2$ and $r>0$ denote by $D(c, r)$ a closed disc centered at $c$ of radius $r$.

Fix a $C^\infty$ function  $\varphi:[0,1]\to [0,1]$ in such a way that the following properties hold:

\vspace{4pt}

1) $\varphi$ is non-increasing;

\vspace{4pt}

2) $\varphi|{[0,1/4]}\equiv 1$;

\vspace{4pt}

3) $\varphi|{[3/4,1]}\equiv 0$.

\vspace{4pt}

{For any $c\in \mathbb{R}^2$, $r>0$, and $q>0$ denote by $\psi(c, r, q)$ the diffeomorphism
$$
\psi(c, r, q):D(c, r)\to D(c, r)$$
defined by setting
$$\psi(c,r,q)(c+ \rho e^{i\theta})= c+ \rho e^{i(\theta+q\varphi(\frac{\rho}{r}))}.
$$}

Fix a strictly increasing function $m:\mathbb{N}\to \mathbb{N}$ such that $m(n)$ takes only odd values, and
grows sufficiently fast (for example $m(n)>e^n$ for all $n\in \mathbb{N}$). Set $q_n=\frac{1}{m(n)}$.

For $n\ge 2$ define $f_{\bar\omega}$ on the discs $D(c_n^a, r_n)$ and $D(c_n^b, r_n)$ by setting
\begin{eqnarray*}
f_{\bar\omega}|{D(c_n^a, r_n)}&=&\psi(c_n^a,r_n, \frac{1}{8}q_n)(x)\\
f_{\bar\omega}|{D(c_n^b, r_n)}&=&\psi(c_n^b,r_n, \frac{1}{16}q_n)(x).
\end{eqnarray*}
Then on the discs $D(c_n^a, r_n/4)$ and $D(c_n^b, r_n/4)$ $f_{\bar\omega}$ is a rotation by $\frac{1}{8}q_n(x)$ and
$\frac{1}{16}q_n(x)$ respectively. Hence they have periodic points of different orders. It follows easily that
$f_{\bar\omega}|D(c_n^a, r_n)$ and $f_{\bar\omega}|D(c_n^b, r_n)$ are not conjugate by homeomorphisms, nor are the restrictions of $f_{\bar\omega}$ to discs of different radii conjugate.

On the discs having $z_n^a$ and $z_n^b$ as centers and $n\ge 2$ we define
\begin{equation*}
f_{\bar\omega}|{D(z_n^a, r_n)}=
\begin{cases}
\psi(z_n^a, r_n, q_n) &\mbox{ if }\omega_{n-1}=0\\
\psi(z_n^a, r_n, q_n/2)&\mbox{ if }\omega_{n-1}=1
\end{cases}
\end{equation*}
and
\begin{equation*}
f_{\bar\omega}|{D(z_n^b, r_n)}=
\begin{cases}
\psi(z_n^b, r_n, q_n/2)&\mbox{ if }\omega_{n-1}=0\\
\psi(z_n^b, r_n, q_n) &\mbox{ if }\omega_{n-1}=1.
\end{cases}
\end{equation*}

The reason for these definitions is that  under the map $f_{\bar{\omega}}$
	\begin{itemize}	
	\item points within radius $r_n/4$ of $c_n^a$ have period $8m_n$
	\item points within radius $r_n/4$ of $c_n^b$ have period $16m_n$
	\item if $\omega_{n-1}=0$, then points within radius $r_n/4$ of $z_n^a$ have period $m_n$ and points within radius $r_n/4$ of $z_n^b$ have period $2m_n$
	\item if $\omega_{n-1}=1$, the periodicities of the points within radius $r_n/4$ of $z^a_n$ and $z_n^b$ are reversed.
	\end{itemize}

Thus the restriction of $f_{\bar\omega}$ to the $z_n^a$ discs are not conjugate to the restriction of $f_{\bar\omega}$ to the $z_n^b$ discs, nor to restrictions of $f_{\bar\omega}$ to discs of different radii or discs centered at the $c^a_n$'s or $c^b_n$'s.

Everywhere but at the points $a$ and $b$ the map $f_{\bar \omega}$ is $C^\infty$ by definition. It is not hard to see that the points $a$ and $b$ are flat fixed points of $f_{\bar \omega}$ (since $q_n$ is exponentially small). Therefore $f_{\bar \omega}$ is a $C^\infty$ diffeomorphism of $\mathbb{R}^2$.

Let us show that

\vspace{4pt}

 I. If $\bar \omega\sim_{E_0} \bar \omega'$, then $f_{\bar\omega}$ is topologically conjugated to $f_{\bar\omega'}$.

\vspace{4pt}

II. If  $f_{\bar\omega}$ is topologically conjugated to $f_{\bar\omega'}$, then $\bar \omega\sim_{E_0} \bar \omega'$.

\vspace{4pt}

Let us start with the first statement. If $\bar \omega\sim_{E_0} \bar \omega'$, then for some $k\in \mathbb{N}$, $\omega_n=\omega_n'$ for all $n\ge k$. The following statement is {immediate}.

\begin{lemma}
There is a homeomorphism $h:\mathbb{R}^2\to \mathbb{R}^2$ such that

\vspace{4pt}

a) $h|{D}=Id$ for all discs of radius $r_n$ centered at $c_n^a$ or $c_n^b$,  $n\ge2$;

\vspace{4pt}

b) $h|D$ is an isometry of discs $D$ of radius $r_n$ centered at $z_n^a$ or $z_n^b$,  all $n\ge2$;

\vspace{4pt}

c) for each $n\le k$ one has

\vspace{4pt}

$\qquad$ $h|{D(z_n^{a}, r_n)}$, and $h|{D(z_n^{b}, r_n)}$  are the identity maps if
$\omega_n=\omega_n'$ and
\vspace{3pt}

$\qquad$ $h[{D(z_n^{a}, r_n)}]=D(z_n^{b}, r_n)$,  and $h[{D(z_n^{b}, r_n)}]=D(z_n^{a}, r_n)$ if
$\omega_n\ne\omega_n'$;

\vspace{4pt}

d) for  $n>k$, $h|D(z_n^{a}, r_n))$ and  $h|D(z_n^{b}, r_n))$ are the identity maps.

\end{lemma}

Now it is easy to see that
$$
h\circ f_{\bar \omega}=f_{\bar \omega'}\circ h.
$$

Thus if $\bar \omega\sim_{E_0} \bar \omega'$, then $f_{\bar\omega}$ is topologically conjugated to $f_{\bar\omega'}$.

\bigskip

\noindent For the converse, let us assume now that $f_{\bar\omega}$ is topologically conjugated to $f_{\bar\omega'}$, i.e. there exists a homeomorphism $h:\mathbb{R}^2\to \mathbb{R}^2$ such that
$$
h\circ f_{\bar \omega}=f_{\bar \omega'}\circ h.
$$

The homeomorphism $h$   must send the set of periodic points of order $p$ for $f_{\bar\omega}$  to the
set of periodic points of order $p$ for $f_{\bar{\omega}'}$. Hence $h$ sends the compliment of the union
of the discs $D(z_n^{a}, \frac{3}{4}r_n)$,  $D(z_n^{b}, \frac{3}{4}r_n)$, $D(c_n^{a}, \frac{3}{4}r_n)$ and
$D(c_n^{a}, \frac{3}{4}r_n)$ as well as the
 collection  of points $\{z_n\}, \{c_n\}$ to itself.
In particular this implies that $h$ must map the discs $D(z_n^a, \frac{3}{4}r_n)$, $D(z_n^b, \frac{3}{4}r_n)$,
$D(c_n^a, \frac{3}{4}r_n)$, $D(c_n^b, \frac{3}{4}r_n)$ to each other one by one.

Since $h$ preserves the period of points $D(c_n^{a}, \frac{3}{4}r_n)$ must be sent to $D(c_n^{a}, \frac{3}{4}r_n)$  and  similarly for the $D(c_n^b, \frac{3}{4}r_n)$'s.  Because the $c_n^a$'s converge to $a$ and the $c_n^b$'s converge to $b$, $h$ is the identity on $\{a, b\}$.

Finally $h$ if $\bar{\omega}_n=\bar{\omega}_n'$ then $h$ must send the discs
$D(z_n^a, \frac{3}{4}r_n)$ and $D(z_n^b, \frac{3}{4}r_n)$ to themselves and transpose them if $\bar{\omega}_n\ne \bar{\omega}_n'$.

By continuity, for all but finite number of $n\ge 2$, we must have
$$
h(D(z_n^{a}, \frac{3}{4}r_n))=D(z_n^{a}, \frac{3}{4}r_n), \ \ \text{and}\ \ \ h(D(z_n^{b}, \frac{3}{4}r_n))=D(z_n^{b}, \frac{3}{4}r_n).
$$
Hence, $\bar \omega \sim_{E_0} \bar \omega'$. The proof of Theorem \ref{l.1} is complete.
\end{proof}

\begin{coro}\label{E0anydim} Let $M$ be a $C^\infty$ manifold of dimension at least 2.  Then there is a continuous map
$\mathbf{F}:\Sigma\to \text{Diff}^\infty(M)$ such that $\mathbf{F}(\bar\omega)$ is topologically conjugate to $\mathbf{F}(\bar \omega')$ if and only if $\bar\omega\sim_{E_0} \bar\omega'$.
\end{coro}
\pf It suffices to show this for a bounded open ball of $\mathbb R^n$ for $n\ge 2$. It contains an open subset diffeomorphic to a product of two-dimensional disc and a torus of dimension $n-2$. The map $\mathbf{F}$ can be constructed in such a way that $\mathbf{F}(\bar\omega)$  is identity away from that open set, and on that open set it acts as a product of a map on the disc that was constructed above, in the proof of Theorem \ref{l.1}, and identity map on the $(n-2)$-dimensional torus.\qed

\begin{remark}\label{on the boundary}
 The analog of Theorem \ref{dim two} holds even  if we restrict ourselves to the boundary of the set structurally stable, or even stronger, of the set of Morse-Smale diffeomorphisms.

 One can show that exact construction of $\mathbf{F}$ described above can be carried out on the boundary of the set of Morse-Smale diffeomorphisms, but it becomes more transparent in case we make a slight modification in the construction. Namely, rather than  the construction above define the map $F(\bar \omega)$ in the discs $D(c_n^a, r_n), D(c_n^b, r_n), D(z_n^a, r_n)$, and $D(z_n^b, r_n)$ as time one shifts of $C^\infty$-gradient vector fields that vanish outside of a corresponding disc, decay exponentially fast as $n$ goes to $\infty$, and have pairwise non-homeomorphic sets of critical points (e.g. as in Fig. \ref{MScritical}). The argument for $\mathbf{F}$ being a reduction remains substantially the same.

\begin{figure} \begin{center} \includegraphics[width=0.6\textwidth]{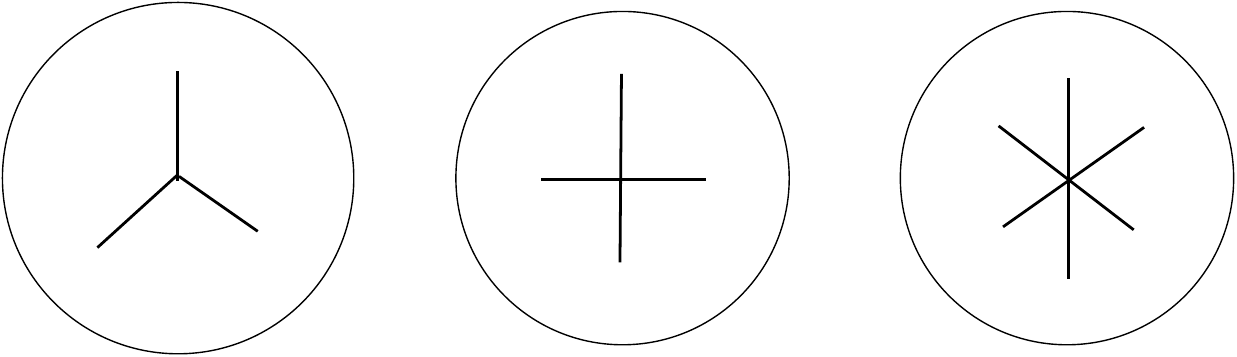} \end{center}
 \caption{Non-homeomorphic sets of critical points inside of the discs.}  \label{MScritical}
\end{figure}

With this change the diffeomorphisms in the range of the reduction  belong to the  boundary of the set of 
Morse-Smale diffeomorphisms. Indeed, the set of functions $L$ satisfying the first condition of the construction 
of Morse-Smale diffeomorphisms given in Remark \ref{r.MS}  forms an open and dense subset in the space of 
$C^\infty$ functions, and the  time one shift along a gradient vector field generated by such a function $L$ must 
be a Morse-Smale diffeomorphism. Therefore, the time one shift along a gradient vector field (in particular, the 
diffeomorphisms in the image of $\mathbf{F}$ constructed in this way) must be on the boundary of the set of 
Morse-Smale diffeomorphisms.

 This remark is related to the open question \ref{q.MS} in Section \ref{s.questions}. 
 \end{remark}

\section{Five dimensional diffeomorphisms}\label{s.5}

Here we prove Theorem \ref{main result}, and hence Theorem \ref{dim at least five}.

\subsection{Statement of the result}

Recall that the space \hyperlink{graphs}{$Graphs$} is given by the collection of symmetric maps $\mathbb{N}\times \mathbb{N}\to \{0,1\}$ endowed with a  topology that makes it homeomorphic to a Cantor set.

\begin{theorem}\label{t.main}
There exists a continuous map 
\[\mathbf{R}:Graphs\to \text{Diff}^\infty(\mathbb{R}^5)\]
 such that $\mathbf{R}(E_1)$ and $\mathbf{R}(E_2)$ are topologically conjugate if and only if the graphs $G_{E_1}$ and $G_{E_2}$ are isomorphic.
\end{theorem}

\begin{remark}\label{r.gen}
In Theorem \ref{t.main}, $\text{Diff}^\infty(\mathbb{R}^5)$ can be replaced by $\text{Diff}^\infty(M)$ for any smooth manifold of dimension at least 5. Indeed, in our construction any diffeomorphism in the range of $\mathbf{R}$ is an identity outside of a finite ball, and therefore one can replace $\mathbb{R}^5$ by a single chart of a manifold $M$.  This is explained further in  Remark \ref{r.compact}.
\end{remark}

\subsection{Construction of the map $\mathbf{R}$}

The idea is to embed a complete countable graph into a bounded domain in $\mathbb{R}^5$ and to associate a given countable graph with a diffeomorphism that is supported in a union of small domains around the edges of the embedded graph. Each such domain is either a basin of attraction of a segment of fixed points (for the edges of the graph) or a repelling domain (for those edges of the complete graph that are missing in the given one).

Let us start by placing the vertices of the complete countable graph into $\mathbb{R}^5$. In fact, we will place them all into a three dimensional subspace, and the extra two dimensions will be used to construct a topological conjugacy between the diffeomorphisms defined by isomorphic graphs. For this purpose we adopt the following
 \hypertarget{Es}{notation}:

 We denote the subset of $\mathbb R^5$ given by $\{(x_1, x_2, x_3, 0, 0):x_i\in \mathbb R\}$ by ${\mathbb E}_3$.  Its orthogonal complement, $\{(0,0,0, x_4, x_5):x_i\in \mathbb R\}$ is denoted ${\mathbb E}_2$.  
 \hypertarget{defofr}{Note that ${\mathbb E}_3\perp {\mathbb E}_2$} and {they span $\mathbb R^5$.}

\begin{prop}\label{p.1}
In ${\mathbb E}_3$ there exists a bounded sequence $\{x_n\}_{n\in \mathbb{N}}$, a point $r\notin \{x_n\}_{n\in \mathbb{N}} $ and a sequence of positive numbers $\{\eta_n\}$ such that

\

1) $x_n\ne x_m$ if $n\ne m$;

\

2)  $\text{\rm dist}(x_n, r)< 2^{-n-1}$, so the sequence  $\{x_n\}_n$ converges to $r$,

\

3) No three points from the set $\{r, \{x_n\}_{n\in \mathbb{N}}\}$ are on the same line;

\

4) No four points from the set $\{r, \{x_n\}_{n\in \mathbb{N}}\}$ are on the same two dimensional plane.

\

5) The sequence of directions in the tangent space $T_r\mathbb{R}^5$ given by vectors $\{x_n-r\}$ converges to some direction that does not belong to this sequence.

\

\hypertarget{ImnandIn}{6)} Let \hypertarget{In}{$I_{m,n}$} be the closed interval with the end points $x_m$ and $x_n$, and  $\mathbf{I}_n$ be the closed interval with the end points $x_n$ and $r$. For any $n\ne k, m$, and  any point $y\in I_{k,m}$ we have $\angle ([r, y], \mathbf{I}_n)>\eta_n$.

\

\end{prop}
\pf Fix a point $r\in {\mathbb E}_3$ and a ray with end point $r$. Let $x_1$ be any point not on the ray. If $\{x_1, x_2, ..., x_n\}$ are constructed, let us choose $x_{n+1}\in {\mathbb E}_3$ in such a way that for $\{x_1, \ldots, x_{n+1}\}$ the properties 1)-4) are satisfied, and the angle between the initial ray and an interval $[r, x_{n+1}]$ is positive but smaller than $2^{-n-1}$. A sequence $\{x_1, x_2, \ldots\}$ constructed in this way satisfies all the required properties.
\qed

\noindent Notice that the properties in Proposition \ref{p.1} imply also the following:

\

7) We have
$$
I_{m,n}\cap I_{k,l}=\emptyset\ \ \ \text{if} \ \ \{k,l\}\cap \{m, n\}=\emptyset,
$$
and $I_{m,n}\cap I_{k,l}$ consists of at most one point (that must be an end point of each interval) if $\{m,n\}\ne \{k,l\}$,

\

8) The interval $\mathbf{I}_n$ is disjoint from $I_{m,k}$ if $n\ne m,k$.

\

\noindent \hypertarget{bfI}{From now on} we fix a sequence of vertices $\{x_n\}$ provided by Proposition \ref{p.1}. A picture of the lines $I_{m,n}$ and $\mathbf{I}_n$ is given in figure \ref{basic3Dpicture}.

\begin{definition}\label{bigI}
Let $\mathbf{I}=\bigcup_{n}\mathbf{I}_{n}$.
\end{definition}
\noindent  Notice that $\mathbf{I}$ is a compact set, and that
\begin{eqnarray*}\overline{\cup_{m,n}I_{m,n}}&=&\mathbf{I}\cup\left(\cup_{m,n}I_{m,n}\right), \\ \mathbf{I}\cap\left(\cup_{m,n}I_{m,n}\right)&=&\{x_n\}_{n\in \mathbb{N}}.
\end{eqnarray*}

\begin{remark}\label{r.compact}
The points $\{x_n\}_n$ represent the vertices of a countable graph and the interval $I_{m,n}$ connects $x_n$ to $x_m$. The intervals $I_{m,n}$ represent the edges of the countable complete  graph on $\{x_n\}_n$. All the diffeomorphisms, homotopies, and conjugacies that we will consider are  supported inside a ball of radius 1 around the point $r$. This justifies Remark \ref{r.gen} above.
\end{remark}

	\begin{figure} 
	\begin{center}
\includegraphics[width=0.4\textwidth]{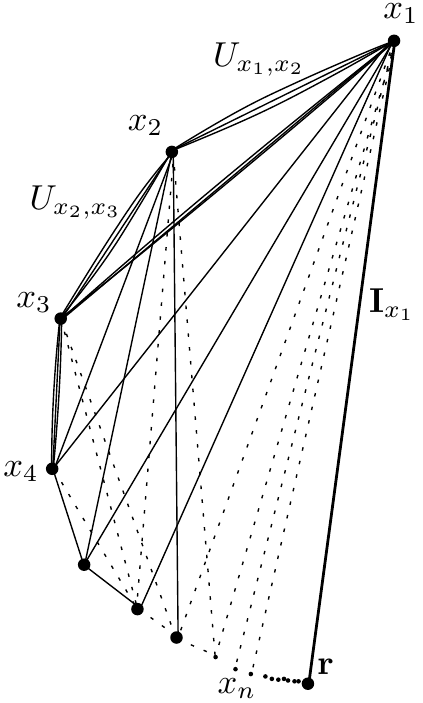}
\end{center}
		\caption{{The vertices of an embedded graph.}}
		\label{basic3Dpicture}
	\end{figure}

\hypertarget{alphamn}{Fix $m, n$ and set}
\begin{equation}\notag
\alpha_{m,n}=\frac{1}{1000}\inf_{\{k, l\}\cap \{n,m\} =\emptyset}\min\{\angle(I_{m,k}, I_{m,n}), \angle(I_{n,k}, I_{m,n}), \text{dist}(I_{m,n}, I_{k,l})\}.
\end{equation}
\begin{lemma} For a sequence $\{x_n\}$ as above, for all $m, n$,  $\alpha_{m,n}>0$.
\end{lemma}
\pf If not there would be a sequence of $k$'s  and $l$'s going to infinity such that either
\begin{enumerate}[a.)]
\item for one of $m$ or $n$, $\angle(I_{m,k}, I_{m,n}) \to_k 0$ or $\angle (I_{n,k}, I_{m,n})\to_k 0$ or
\smallskip

\item $dist(I_{m,n}, I_{k,l}) \to_{k,l} 0$
\end{enumerate}
In both cases, because the $x_m$ and $x_n$ are fixed and $x_k, x_l$ converge to $r$,  the minimum of the sequences in a.) or b.) is achieved by some $k,l$ in the sequence.\qed

Let us denote by $\psi$  the function $\psi:[-1,1]\to \mathbb{R}$ given by
	\begin{equation}\psi(x)=\frac{1-x^2}{2}.
	\label{psidef}
	\end{equation}
 Notice that
$\psi$ is an upside down parabola symmetric about $x=0$ with

\begin{enumerate}

\item  $0<\psi(x)\le\frac{1}{2}$ if $x\in (-1,1)$, and $\psi(-1)=\psi(1)=0$;

\item $\psi$ is convex;

\item $\psi'(-1)=1$, $\psi'(1)=-1$.
\end{enumerate}

\begin{definition}\label{cigars} \hypertarget{cigars}{We} now give the definitions used  to define the supports of our diffeomorphisms.
\begin{description}
\item[Central Axes] Given $x,y\in \mathbb{R}^5$, $x\ne y$, denote by $\gamma_{x,y}$ the line  connecting $x$ and  $y$ given by
$$
\gamma_{x,y}(t)=\frac{x+y}{2}+t\frac{y-x}{2}.
$$
Note that $\gamma_{x,y}(-1)=x$ and $\gamma_{x,y}(1)=y$.
\item[Normal Hyperplanes] For $t\in \mathbb R$, let $\Pi_{x,y}(t)$ the affine 4-dimensional subspace that contains the point $\gamma_{x,y}(t)$ and is orthogonal to the line $\gamma_{x,y}$.
\item[Atomic Domains (``cigars'')]
\hypertarget{Us}{Denote by} $U_{x,y, \varepsilon}\subset \mathbb{R}^5$ an open set such that $\Pi_{x,y}(t)\cap U_{x,y, \varepsilon}=\emptyset$ if $|t|\ge 1$, and
 $\Pi_{x,y}(t)\cap U_{x,y, \varepsilon}$ is an open ball in $\Pi_{x,y}(t)$ of radius $\varepsilon \psi(t)$ centered at $\gamma_{x,y}(t)$ if $|t|< 1$.
 \end{description}
\end{definition}

\hypertarget{gammamn}{To simplify the notation, let us also denote}
$$
\gamma_{m,n}=\gamma_{x_m,x_n}, \ \ \Pi_{m,n}(t)=\Pi_{x_m,x_n}(t), \ \ and \ \ U_{m,n, \varepsilon}=U_{x_m, x_n, \varepsilon}.
$$
\noindent We will occasionally informally refer to $U_{m,n,\varepsilon}$ as the  $\varepsilon$-\emph{cigar} between $x_m$ and $x_n$.

\begin{figure}[htp] \hfill\includegraphics[width=0.9\textwidth]{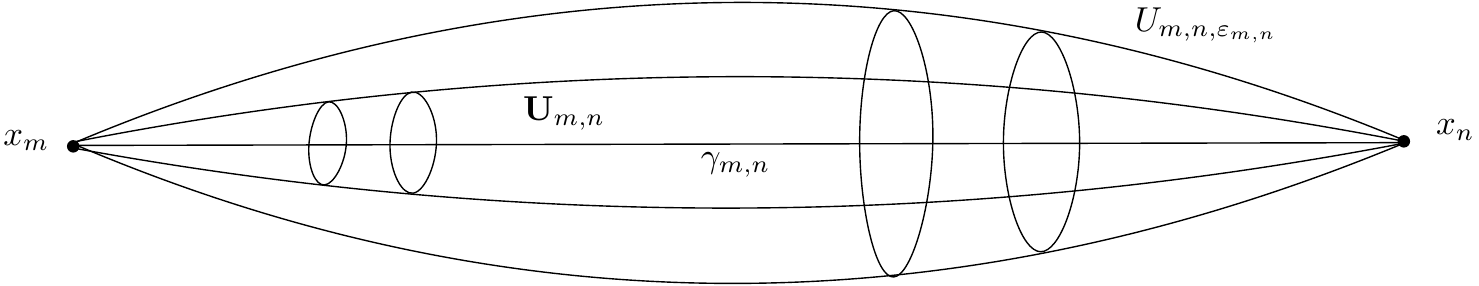}\hspace*{\fill}
		\caption{{``Cigars'' $U_{x_m, x_n, \varepsilon_{m,n}}$ and $\mathbf{U}_{m,n}$.}}
		\label{Cigarpicture}
\end{figure}

\hypertarget{U no epsilon}{Let us set
\begin{equation}
\varepsilon_{m,n} =\min{(0.001\alpha_{m,n}, 2^{-n-m})},
\end{equation}
and denote $\mathbf{U}_{m,n}=U_{x_m, x_n, \frac{1}{2}\varepsilon_{m,n}}$.}  See figure  \ref{Cigarpicture}
for illustration.
\begin{remark}
\hypertarget{Uprimes}{Later} in the proof we will also consider cigars of smaller sizes $\varepsilon_{m,n}'$ and 
$\varepsilon_{m,n}''$, so that $\varepsilon_{m,n}''\ll\varepsilon_{m,n}'\ll \varepsilon_{m,n}$, and therefore 
$U_{m,n, \varepsilon_{m,n}''}\subset U_{m,n, \varepsilon_{m,n}'}\subset  U_{m,n, \varepsilon_{m,n}}$. We will 
also use the notation $\mathbf{U}_{m,n}'=U_{x_m, x_n, \frac{1}{2}\varepsilon_{m,n}'}$ and 
$\mathbf{U}_{m,n}''=U_{x_m, x_n, \frac{1}{2}\varepsilon_{m,n}''}$.
\end{remark}
The fact that  $\alpha_{m,n}>0$ guarantees that the following statement holds:
\begin{prop}\label{p.3}
For each $\{m, m\}$, $\{k, l\}$:
\begin{enumerate}

\item $U_{m,n, \varepsilon_{m,n}}\cap U_{k,l, \varepsilon_{k,l}}=\emptyset$ if $\{m,n\}\ne \{k,l\}$;

\item If $m=k$, $n\ne l$, then 
$$
\overline{U_{m,n, \varepsilon_{m,n}}}\cap \overline{U_{k,l, \varepsilon_{k,l}}}=x_m;
$$

\item If the edges $I_{m,n}$ and $I_{k,l}$ are disjoint, then the sets $\overline{U_{m,n, \varepsilon_{m,n}}}$ and $\overline{U_{k,l, \varepsilon_{k,l}}}$
 are also disjoint.
 \end{enumerate}
\end{prop}

\paragraph{\bf The dynamics on the atomic
domains $\mathbf{U}_{m,n}$.}
We start by defining maps $f^+_{m,n}$ and $f^-_{m,n}$, that serve as basic building blocks
for the final diffeomorphisms.  They each send
$\mathbf{U}_{m,n}$ to itself and are the identity on the complement of $\mathbf{U}_{m,n}$. The map $f^+_{m,n}$ smoothly moves the elements of $\mathbf{U}_{m,n}$ \emph{towards} the centerline 
$\gamma_{m,n}$ while preserving $\Pi_{m,n}$, and $f^-$ moves the elements $\mathbf{U}_{m,n}$ towards the outer boundary (away from the centerline) while preserving $\Pi_{m,n}$. 
\smallskip

{\noindent{\bf In more detail:}}

\begin{enumerate}
\item Let $\lambda_{m,n}:\mathbb{R}^5\to \mathbb{R}$ be $C^\infty$ function such that $\lambda>0$ in $\mathbf{U}_{m,n}$, $\lambda=0$ in $\mathbb{R}^5\backslash {\bf U}_{m,n}$, and $\|\lambda_{m,n}\|_{\infty}\le 1$.

\item Let $w_{m,n}$ be the vector field orthogonal to $\gamma_{m,n}$, directed towards $\gamma_{m,n}$, and such that $\|w_{m,n}(x)\|=\text{dist}\,(x,\gamma_{m,n})$.

\hypertarget{flatfixed}{.}
\item \label{propsoftheta} Let $\Theta\in C^{\infty}(\mathbb{R}^5)$ be a smooth function such that
	\begin{enumerate}
	\item $\Theta(x)\in [0,1]$ for all $x\in \mathbb{R}^5$;

	\item $\Theta(x)=0$ if and only if $x\in $\hyperlink{bfI}{$\mathbf{I}$}; 

	\item Every point of the set $\mathbf{I}$ is a flat zero point of $\Theta$.\label{thetaisflat}
	\end{enumerate}
\end{enumerate}
\hypertarget{fpm}{We can now define  $f^+$ and $f^-$}.
\begin{definition}\label{plusnminus}
Let $v_{m,n}$ be  {the} vector field  {on} $\overline{\mathbf{U}_{m,n}}$ given by
\hypertarget{vmn}{\begin{equation}v_{m,n}=\Theta\cdot \lambda_{m,n}\cdot w_{m,n}.\label{vmn}
\end{equation}}
Let {{ $f^+_{m,n}$} be time-1 shift associated with the vector field $v_{m,n}$, and { $f^-_{m,n}$} be time-1 shift associated with the vector field $(-v_{m,n})$.}
\end{definition}

Notice that while the properties above do not define the choice of $\lambda_{m,n}$ and $\Theta$ completely, dynamics of $f^+_{m,n}$ (or $f^-_{m,n}$) does not depend on this choice (up to a continuous change of coordinates), as the following statement shows.

\begin{prop}\label{f+aresame}
Fix two pairs $(m,n)$ and $(m^*, n^*)$ and let $g$ be either $f^+_{m,n}$ or $f^-_{m,n}$ and $h$ be
$f^+_{m^*,n^*}$ or $f^-_{m^*,n^*}$. Then $g$ is topologically conjugate to $h$ if and only either
\begin{itemize}
	\item $g=f^+_{m,n}$ and $h=f^+_{m^*, n^*}$ or
	\item  $g=f^-_{m,n}$ and $h=f^-_{m^*, n^*}$.
\end{itemize}
\end{prop}

\pf  Suppose first that $g=f^+_{m,n}$ and $h=f^+_{m^*, n^*}$.  The case $g=f^-_{m,n}$ and $h=f^-_{m^*, n^*}$ is exactly analogous. Note that it  suffices to show the result for $m^*=0$ and $n^*=1$.

 There is an affine map $\phi$ taking the cigar
 $\overline{U}_{m,n}$ to $\overline{U}_{0,1}$.  The map $\phi$
 conjugates $g$ to a diffeomorphism $g^*$ of
  $\overline{U}_{0,1}$ that only has fixed points on the boundary and on the line $\gamma_{0,1}$.
  Moreover $g^*$ moves points directly towards the centerline through the subspaces $\Pi_{x,y}(t)$. To
  finish showing that $g$ is conjugate to $h$ it suffices to show that $g^*$ is conjugate to $h$ by a map
  $\psi$.

  To show that $g^*$ and $h$ are conjugate we use the method of \emph{fundamental domains}.
Let $V$ be the boundary of $U_{0,1,\frac{\epsilon_{0,1}}{4}}$.  If $p\in V$ then $p$ lies in a hyperplane
$\Pi_{x,y}(t)$  and $g^*(p)$ lies in the same $\Pi_{x,y}(t)$ and on the line segment between $p$ the centerline of $U_{0,1}$. For $j\in \poZ$ let  $\psi((g^*)^j(p))=h^j(p)$.  For each point $q$ belonging to  the half open  line segment between $p$ and $g^*(p)$ let $\psi((g^*)^j(q))=h^j(q)$.
We claim that  $\psi$ is a homeomorphism of $\overline{U}_{0,1}$ that conjugates $g^*$ to $g$. 

Note that for each $q$ not on the centerline, $(g^*)^j(q)$ converges monotonically to 
the centerline and $j$ goes to $\infty$ and to the boundary as $j$ goes to $-\infty$. For all $q\in U_{0,1}$  there is a unique $p\in V$ on the line segment connecting the centerline to the boundary that passes through $q$. There is a unique $k\in \poZ$ such that $(g^*)^k(q)$ lies in the half open line segment between $p$ and $g^*(p)$.  
The analogous facts hold for $h$ as well. 

Hence $\psi$ is well-defined on all of $U_{0,1}$ and since both $g^*$ and $h$ converge monotonically along the line segments to the centerline $\psi$ is a conjugacy.

\medskip
Now suppose that $g=f^+_{m,n}$ and $h=f^-_{m^*,n^*}$ and $\phi$ is a homeomorphism from $\overline{U}_{m,n}$ to $\overline{U}_{m^*, n^*}$ such that $\phi\circ g=h\circ \phi$.

Since $\phi$ must take the fixed points of $g$ to the  fixed points of $h$,  $\phi$ takes
 $\gamma_{m,n}$ to
$\gamma_{m^*, n^*}$. Let $x\in U_{m,n}$ and consider the sequence $\la g^n(x):n\in\nn\ra$.  By compactness this sequence has a limit point $x^*$ in $\bar{U}_{m,n}$. Since $g$ is continuous $g(x^*)=x^*$.  Since $g$ takes points towards
$\gamma_{m,n}$ along a line orthogonal to $\gamma_{m, n}$, $x^*$ must be on $\gamma_{m,n}$ and be the limit of $\la g^n(x):n\in\nn\ra$.

It follows that $\{h^{n}(\phi(x)):n\in\nn\}$ converges to a point on $\gamma_{m^*, n^*}$.  However no point not already on $\gamma_{m^*, n^*}$ has a point on $\gamma_{m^*, n^*}$ as a limit point of the orbit of $f^-_{m^*, n^*}$.  This is a contradiction.
\qed

\paragraph{\bf The definition of the reduction $\mathbf{R}$.}
Recall that the space  \hyperlink{graphs}{\emph{Graphs}} consists of symmetric functions $E:\mathbb{N}\times \mathbb{N}\to \{0,1\}$  that encode a graph $G_E$ with vertices $\nn$ and $n$ is connected to $m$ if and only $E(m,n)=1$.
\smallskip

\noindent Now Theorem \ref{t.main} can be reformulated more explicitly:
\begin{prop}\label{p.reform}

Consider the map
$$
\mathbf{R}:Graphs \to \text{Diff\,}^\infty(\mathbb{R}^5)
$$
be defined so  that $\mathbf{R}(E)$ is a diffeomorphism supported on $\cup_{m,n}\mathbf{U}_{m,n}$, and

\begin{equation}\mathbf{R}(E)|_{\mathbf{U}_{m,n}}=\left\{
                                    \begin{array}{ll}
                                      f^+_{m,n}, & \hbox{if $E(m,n)=1$;} \\
                                      f^-_{m,n}, & \hbox{if $E(m,n)=0$.}
                                    \end{array}
                                  \right.
                                  \label{definitionofdiffeo}
\end{equation}
Then $\mathbf{R}(E_1)$ and $\mathbf{R}(E_2)$ are topologically conjugate if and only if $E_1$ and $E_2$ are isomorphic graphs.
\end{prop}

\begin{remark}
Notice that the map $\mathbf{R}(E)$ defined above is $C^\infty$ diffeomorphism. Indeed, it is clear that it is smooth at the points outside the set of accumulation points of the collection of the sets  $\mathbf{U}_{m,n}$ (which is
 exactly \hyperlink{bfI}
{$\mathbf{I}$}). Property \ref{thetaisflat} of the function $\Theta$ implies
 that the map $\mathbf{R}(E)$ has a flat fixed point at every point of $\mathbf{I}$.
 
 Moreover the map $R$ is a continuous map with respect to the topologies on $Graphs$ and 
 $\text{Diff\,}^\infty(\mathbb{R}^5)$
\end{remark}

Now we need to prove Proposition \ref{p.reform}.

\subsection{The Easy Part}

Here we prove that if $\mathbf{R}(E_1)$ and $\mathbf{R}(E_2)$ are topologically conjugate, then $E_1$ and $E_2$ are isomorphic.

 Suppose $\mathbf{R}(E_1)=F_1$ and $\mathbf{R}(E_2)=F_2$ are topologically conjugate, i.e. there exists a homeomorphism $H:\mathbb{R}^5\to \mathbb{R}^5$ such that $F_2\circ H=H\circ F_1$. In this case $H(\text{Fix}(F_1))=\text{Fix}(F_2)$, and $H(\mathbb{R}^5\backslash \text{Fix}(F_1))=\mathbb{R}^5\backslash \text{Fix}(F_2)$, where $\text{Fix}(F_i)$ is the set of fixed points of the map $F_i$. Moreover, since $H$ is a homeomorphism, it must send every connected component of $\mathbb{R}^5\backslash \text{Fix}(F_1)$ into a connected component of $\mathbb{R}^5\backslash \text{Fix}(F_2)$. Hence $H$ sends $\mathbf{U}_{m,n}$ to $\mathbf{U}_{k,l}$ for some $k,l\in \mathbb{N}$, and $\partial \mathbf{U}_{m,n}$ to $\partial \mathbf{U}_{k,l}$. 
 Moreover, the set of points that belong to the boundary of more than one of the domains $\{\mathbf{U}_{m,n}\}$ must be invariant under $H$, but this is exactly the sequence $\{x_n\}_{n\in \mathbb{N}}$. Therefore $H$ restricted to $\{x_n\}_{n\in \mathbb{N}}$ permutes these points and thus induces a permutation  $\varphi(H)=\varphi:\mathbb{N}\to \mathbb{N}$.

 Let us show that $\varphi$ must be an isomorphism between the graphs $E_1$ and $E_2$. Indeed, if
 $\varphi(n)=k$ and $\varphi(m)=l$, then $H(\mathbf{U}_{m,n})=\mathbf{U}_{k,l}$\ , so $H(\overline{\mathbf{U}_{m,n}})=\overline{\mathbf{U}_{k,l}}$, and $H|_{\overline{\mathbf{U}_{m,n}}}$ conjugates $F_1|_{\overline{\mathbf{U}_{m,n}}}$ and $F_2|_{\overline{\mathbf{U}_{k,l}}}$.

 By Proposition \ref{f+aresame}, if $F_1|_{\overline{\mathbf{U}_{m,n}}}=f^+_{m,n}$ then  $F_2|_{\overline{\mathbf{U}_{k,l}}}=f^+_{k,l}$ and similarly if $F_1|_{\overline{\mathbf{U}_{m,n}}}=f^-_{m,n}$.  It follows that $m$ and $n$ are connected in $E_1$ just in case $k$ and $l$ are connected in $E_2$. Hence $\phi$ is an isomorphism of $E_1$ with $E_2$.\qed

\subsection{Construction of conjugacy}

Suppose now that the graphs $G_{E_1}$ and $G_{E_2}$ are isomorphic, i.e. there exists a permutation $\varphi:\mathbb{N}\to \mathbb{N}$ such that
$$
E_1(m,n)=1\ \ \ \Leftrightarrow \ \ \ E_2(\varphi(m),\varphi(n))=1.
$$
Let us show that diffeomorphisms $F_1=\mathbf{R}(E_1)$ and $F_2=\mathbf{R}(E_2)$ are topologically conjugate.
First we will present the permutation $\varphi:\mathbb{N}\to \mathbb{N}$  as an infinite composition of transpositions (Lemma \ref{l.permrepr} below). Since $\phi$ is an isomorphism between the graphs $E_1$ and $E_2$, this presentation can be viewed as giving a sequence of intermediate graphs $E^{(i)}$. For each transposition 
$\phi_i$, we construct a conjugacy between transformations corresponding to $E^{(i)}$ and  $E^{(i+1)}$.  This is done in a manner that  the infinite sequence of intermediate conjugacies converges to a conjugacy between $\mathbf{R}(E_1)$ and $\mathbf{R}(E_2)$.
 
\begin{remark}
There is a temptation to start with a construction of a homotopy of the graph, and then use general results to extend it to a homotopy of an ambient space. Unfortunately  some simple, but surprising, examples show that such an extension does not have to exist in general,  e.g. see Example 1.14 from \cite{M}. Therefore, we give an explicit construction.
\end{remark}

\paragraph{\bf Decomposing Permutations.}
In order to construct such a conjugacy $H:\mathbb{R}^5\to \mathbb{R}^5$ we will use a representation of the permutation $\varphi:\mathbb{N}\to \mathbb{N}$ as an infinite converging composition of transpositions \begin{equation}\label{e.transpos}
\varphi=\ldots \circ\ldots \varphi_2\circ \varphi_1.
\end{equation}
Each transposition $\phi_i$ will correspond to a homeomorphism $h_i$ and, setting
\[H_n=h_n\circ h_{n-1} \circ \dots h_2\circ h_1\]
we will have $H=\lim_nH_n$.

In order for that infinite composition of $h_i$'s to converge to a homeomorphism we need to choose a representation of a permutation $\varphi:\mathbb{N}\to \mathbb{N}$ to have some special properties. This is what we do now.

Recall that for functions $\varphi, \la \Phi_n:n\in\nn\ra$ mapping from $\nn$ to $\nn$,
\[\varphi=\lim_{n\to \infty}\Phi_n\]
if and only if for all $m$ there is an $N$ for all $n>N, \Phi_n(m)=\Phi_N(m)$ and $\varphi(m)=\Phi_N(m)$. If each $\Phi_n$ is a permutation then $\varphi$ is a permutation if and only if $\lim_{n\to \infty}\Phi_n^{-1}$ also converges (in which case the limit is $\varphi^{-1}$).
\begin{definition}\label{d.contvert}
Suppose that $\varphi:\mathbb{N}\to \mathbb{N}$ is a permutation of the natural numbers, and
$\varphi=\lim_{n\to \infty}\Phi_n$, where
$\Phi_n=\varphi_{n}\circ \varphi_{n-1}\circ \varphi_{n-2}\cdots\circ \varphi_1$ and each $\varphi_i$ is a transposition. For each $m\in \mathbb{N}$ define the {\rm contamination set} $Cont_m$ as
\begin{multline*}
    Cont_m=\{m\}\cup\ \{n\in \mathbb{N}\ |\ \exists k\in \mathbb{N}, \ \text{and}\  \{n_1, \ldots, n_{k+1}\},\ \text{with}\ n_1=m,\\ n_{k+1}=n,  
\text{\rm and}\  i_1<i_2<\ldots<i_k, \ \text{\rm such that}\ \varphi_{i_j} \ \text{\rm permutes}\ n_{j}\ \text{\rm and}\ n_{j+1}\}.
\end{multline*}
\end{definition}

\begin{lemma}\label{l.permrepr}
Let $\varphi:\nn\to \nn$ be a permutation of the natural numbers.  Then there is a sequence of transpositions
$\la \varphi_n:n\in\nn\ra$,  such that the following holds:

\begin{enumerate}[1)]

\item If $\Phi_n=\varphi_{n}\circ \varphi_{n-1}\circ \varphi_{n-2}\cdots\circ 	
		\varphi_1$, then
		\begin{equation} 
	\lim_{n\to \infty}\Phi_n=\varphi.
	\end{equation}

\item Any given transposition is encountered at most once,

\vspace{5pt}
\item For each $m$, there are at most two $i$'s with $\phi_i(m)\ne m$.
	
\vspace{5pt}

	\item For each $m\in \mathbb{N}$  the set $Cont_m$ is contains at most four elements;

\vspace{5pt}

\item Each $k$ belongs to only finite number of the sets $\{Cont_m\}_{m\in \mathbb{N}}$.

\end{enumerate}
\end{lemma}

\pf Every permutation $f$ of the natural numbers induces a countable partition\\
	 $\nn=\bigcup_kA_k$ such that:
	\begin{enumerate}
	\item If $A_k$ is infinite, then for each $m\in A_k$, $A_k=\{\varphi^{n}(m):m\in \poZ\}$
	\item If $A_k$ is finite then $\phi$ induces a single cycle on $A_k$.
	\end{enumerate}
	By dovetailing, one can treat each $A_k$ separately. We start by assuming that we have a single $A_k$ and it is infinite. Since we are ignoring the other elements of the partition we can assume that every element of $\nn$ is moved by $\phi$, hence $A_k=\nn$. Write $\nn$ as the double ended sequence
	\[\nn=\la \dots m_{-n}, m_{-n+1}, \dots m_{-2}, m_{-1}, m_0, m_1, \dots m_{n}, m_{n+1}, \dots \ra,\]
	where $m_{\pm n}=\varphi^{\pm n}(m_0)$, $n\in \nn$.
	
	We construct the $\varphi_n$'s in the following way:
\begin{multline*}
  \varphi_1=(-1, 1), \varphi_2=(0, 1), \varphi_3=(-2, 2), \varphi_4=(-1, 2), \ldots, \\
\varphi_{2k-1}=(-k, k), \varphi_{2k}=(-k+1, k), \ldots
\end{multline*}

It is easy to see that the sequence of transpositions
$\la \varphi_n:n\in\nn\ra$ constructed in this way satisfies  the required properties.
\smallskip

\hypertarget{nimi}{A very similar  construction works for finite cycles.} \qed
\noindent{\bf Notation:} Denote pair of natural numbers transposed by $\phi_i$ by $n_i$ and $m_i$.
\medskip

\hypertarget{contspairs}{Let us} formulate several definitions that we will use later in the text:

\begin{definition}\label{d.contedges}
Given a permutation of the natural numbers $\varphi:\mathbb{N}\to \mathbb{N}$ with a chosen representation $\varphi=\ldots \circ \varphi_n\circ \varphi_{n-1}\circ \ldots \circ \varphi_1$, and a pair $\{n,m\}$, $n, m\in \mathbb{N}$, $n \ne m$, define
$$
Cont(n,m)=\left\{ (n', m')\ |\ n'\in Cont_n, \ m'\in Cont_m\ \right\}.
$$
\end{definition}

\begin{remark}
 We will be constructing a sequence of homeomorphisms such that each homeomorphism will swap a pair of vertices. In order to control the orbits that initiate in a small neighborhood of a vertex, we need to take into account that some points could move to a neighborhood of an image (so that second vertex becomes ``contaminated''), while some points could stay in a small neighborhood of the initial vertex. If one of the homeomorphisms from our sequence moves one of the previously ``contaminated'' vertices, some points from its neighborhood can be moved to a neighborhood of another vertex, which is therefore also becomes  ``contaminated''. Definition \ref{d.contvert} defines the maximal set of vertices $Cont_n$  that can become ``contaminated'' in this process, starting with the vertex $n$. Definition \ref{d.contedges} serves a similar purpose, but for edges instead of vertices.
\end{remark}

\begin{definition}\label{d.Stablei}
Let us say that $n\in \mathbb{N}$ is frozen at stage $i$, if $\varphi_j(n)=n$ for any $j\ge i$.
\end{definition}

\begin{definition}\label{d.JTn}
Define
$$
J(n)=\min\{j\in \mathbb{N} \ |\ \forall m\in Cont_n, m \ \text{is frozen at stage $j$}\}
$$
and
$$
T(n)=\max\{n_i, m_i \ |\ i\le J(n)\},
$$
where $n_i, m_i$ are the elements of $\mathbb{N}$ that are swapped by $\varphi_i$.
\end{definition}

Let $d_i$ be the distance between $x_{m_i}$ and
$x_{n_i}$ (where $n_i$ and $m_i$ are the pair swapped by $\phi_i$. By property 2 of Proposition \ref{p.1} and the triangle inequality,
$d_i\le 2^{-n_i-1}+2^{-m_i-1}$. Property 3 of the sequence of transpositions $\{\varphi_i\}$ provided by Lemma \ref{l.permrepr} together with the choice of the sequence of vertices $\{x_n\}\subset \mathbb{R}^5$ given by Proposition \ref{p.1} imply the following:
\begin{lemma}\label{l.sumofdi}
$$
\sum_{i=1}^\infty
d_i<\infty$$
\end{lemma}
\pf
We have
$$
d_i=\text{dist}\,(x_{n_i}, x_{m_i})\le 2^{-n_i-1}+2^{-m_i-1}\le 2^{-\min(n_i, m_i)}.
$$
By property 3 of  Lemma \ref{l.permrepr}), $(n_1, m_1), (n_2, m_2), \ldots$ each natural number appears at most 2 times.  This implies that
$$
\sum_{i=1}^\infty d_i\le \sum_{k=1}^\infty 2\cdot 2^{-k}=2.
$$
\qed

\paragraph{\bf Interpolating graphs.}
Suppose we are given symmetric maps \\ $E_1:\nn\times \nn\to \{0,1\}$ and $E_2:\nn\times \nn\to \{0,1\}$ and a permutation $\varphi:\nn\to \nn$
giving an isomorphism between the corresponding graphs. 
 Let us consider the sequence of graphs coded by
$
E^{(1)}, E^{(2)}, \ldots, E^{(i)}, \ldots
$
such that $E^{(1)}=E_1$, and for $i\ge 1$ and any $m,n\in \mathbb{N}$ we have
$$
 E^{(i+1)}(m,n)= E^{(i)}(\varphi^{-1}_i(m), \varphi^{-1}_i(n))=E^{(1)}((\varphi_i\ldots \varphi_1)^{-1}(m), (\varphi_i\ldots \varphi_1)^{-1}(n)).
$$
Since $\varphi_i\circ\ldots \circ \varphi_1(n)=\varphi(n)$ for all sufficiently large $i$, for given $m,n\in \mathbb{N}$ and large $i$ we also have
\begin{multline*}
E_2(\varphi(m), \varphi(n))=E_1(m,n)=E^{(1)}(m,n)= \\
E^{(i+1)}(\varphi_i\circ \dots \circ \varphi_1(m), \varphi_i\circ \dots \circ \varphi_1(n))=E^{(i+1)}(\varphi(m), \varphi(n)).
\end{multline*}
Hence the sequence of graphs $\{E^{(i)}\}$ converges to $E_2$.

Recall $R$ is the map defined in Proposition \ref{p.reform}.
Denote $G_i=\mathbf{R}(E^{(i)})\in \text{Diff}^\infty(\mathbb{R}^5)$. Then $G_1=F_1=\mathbf{R}(E_1)$, and $G_i$ converge to $F_2=\mathbf{R}(E_2)$ in ${C^0}$-topology\footnote{In fact, in our construction a sequence $\{G_i\}$ converges to $F_2$ in ${C^\infty}$-topology; we are not going to use it.}. We will complete the proof of Theorem \ref{t.main} by proving the following statement:
\begin{prop}\label{p.key}
There exists a sequence of homeomorphisms $\{h_i\}_{i\in \mathbb{N}}$, $h_i:\mathbb{R}^5\to \mathbb{R}^5$, $i=1, 2, \ldots, $ such that:
\vspace{5pt}

1) for each $i=1, 2, \ldots$ we have $h_i\circ G_i=G_{i+1}\circ h_i$, i.e. the $h_i$ is a topological conjugacy between  diffeomorphisms  $G_{i+1}$ and $G_i$;

\vspace{5pt}

2) one has $\sum_{i\in \mathbb{N}}\text{dist}_{C^0}(id, h_i)<\infty$ and

\vspace{5pt}

3) for any $x\in \mathbb{R}^5$, there exists $N=N(x)\in \mathbb{N}$ such that for any $n\ge N$ one has
$$
h_n\circ\ldots \circ h_1(x)=h_N\circ\ldots \circ h_1(x).
$$
\end{prop}

Notice that Proposition \ref{p.key} implies that the maps $F_1$ and $F_2$ are topologically conjugate.
Indeed, the property $2)$ implies that the sequence of compositions $h_n\circ\ldots\circ h_1$ converges
as $n\to \infty$ to a continuous map $H:\mathbb{R}^5\to \mathbb{R}^5$ in the space
$C^0(\mathbb{R}^5, \mathbb{R}^5)$. The property $3)$ implies that $H$ is one-to-one, and because it is
a continuous bijection whose support is contained in a compact subset of $\mathbb R^5$,  it must be a
homeomorphism. Finally, the property $1)$ implies that
$$
h_n\circ\ldots\circ h_1\circ G_1=G_{n+1}\circ h_n\circ\ldots \circ h_1,
$$
and, since $G_{n+1}$ converges to $F_2$ as $n\to \infty$,  we get $H\circ F_1=F_2\circ H$, i.e. the homeomorphism $H$ conjugates the maps $F_1$ and $F_2$. This concludes the proof of Theorem \ref{t.main} modulo Proposition \ref{p.key}.

\

Now we need to prove Proposition \ref{p.key}. To motivate the construction we list in advance some properties of the $h_i$'s. 
	\begin{enumerate}
	\item Each $h_i$ is equal to $\phi_i$ on the set of vertices $\{x_n\}_n$.
	\item Each $h_i$ takes cigars to cigars. 
	\item Every $x\in \mathbb R^5$ is only moved by finitely many $h_i$.
	\end{enumerate}

\subsubsection{\bf Initial choices and building blocks: definition of $\{B_{Q_i}\}$, $\{V_{k, i}\}$, and $\{W_i\}$.}\label{ss.BVW}

\hypertarget{rhon}{We start by choosing the neighborhoods that will contain the supports of the homeomorphisms $h_i$.}

\medskip
\paragraph{\bf Choice of neighborhoods of the vertices.} Choose a sequence of balls $\{B_{x_n}\}$, such that $B_{x_n}$ is centered at $x_n$, and radius of $B_{x_n}$ is smaller than
\begin{equation}\label{e.rho}
\rho_n=\frac{1}{1000}\min_{k,m\ne n}\left(\text{dist}(x_n, x_m), \ \text{dist}(x_n, I_{m,k})\right).
\end{equation}

\paragraph{\bf Choice of vectors $\{w_i\}$ in $\mathbb{E}_2$.}
Let us recall that we are given a sequence of transpositions $\varphi_1, \varphi_2, \ldots $, such that if $\Phi_n=\varphi_n\circ \ldots \varphi_2\circ \varphi_1$ then
$\varphi=\lim_{N\to \infty}\Phi_N$, and  $\varphi_i$ is the transposition of $m_i$ and $n_i$.
 Let us assign to each $\varphi_i$ a positive number $0<q_i<\frac{1}{2}$. We can do so in such a way that for any $i\in \mathbb{N}$ we have $0<q_{i+1}<q_i$. We also require $q_i\to 0$ as $i\to \infty$.

Set $w_i=(\cos \pi q_i, \sin \pi q_i)$. Notice that $w_i\to (1,0)\in \mathbb{R}^2 \cong {\mathbb E_2}$ as $i\to \infty$. {Recall} that we consider the whole space $\mathbb{R}^5$ as $\mathbb{R}^5={\mathbb E_3}\oplus{\mathbb E_2}$, with ${\mathbb E_3}\times {\{(0,0)\}}$ being the hyperspace where all the points $\{x_n\}$ belong, and $\{(0, 0, 0)\}\times {\mathbb E_2}$ being the space where the vectors $w_i$ belong.  Note \hypertarget{conesK}{that the $(0, 0, 0, w_i)$ are orthogonal} to every element of $\mathbb E_3$.

\begin{definition}\label{d.coneK}
Denote by $\beta_j>0$ the minimum among all the angles between $w_j$ and any other vector  $w_i$, $i\ne j$. Denote by $K_j$ the cone in $\mathbb{E}_2$ with central axis the line containing  the origin and $w_j$ and having angle $\frac{1}{3}\beta_j$.
\end{definition}

We now define the sets that will contain all points that will be moved by $h_i$. In order to do so we start by defining some ``building blocks'' that will be used for each $h_i$, and locate them in $\mathbb{R}^5$ in such a way that the only intersections are for  obvious reasons (e.g. two of them contain the same vertex or the same edge of an embedded graph). This is where we need to use the additional dimensions in $\mathbb{R}^5=\mathbb{E}_3\oplus \mathbb{E}_2$, and our choice of the vectors $\{w_i\}\subset \mathbb{E}_2$.

\medskip

\hypertarget{QandBQ}{\noindent{\bf Defining $\{Q_i\}$ and $\{B_{Q_i}\}$.}}
Recall that $d_i$ is the distance between $x_{m_i}$ and $x_{n_i}$ where $\phi_i$ transposes $m_i$ and $n_i$. Denote by $Q^u_i$ the union of three line segments,
$$
Q^u_i = I_{x_{n_i}, x_{n_i}+w_id_i}\cup  I_{x_{n_i}+w_id_i, x_{m_i}+w_id_i}\cup I_{x_{m_i}+w_id_i, x_{m_i}},
$$
and $Q^d_i$ the union of the corresponding  line segments:
$$
Q^d_i=I_{x_{m_i}, x_{m_i}-w_id_i}\cup I_{ x_{m_i}-w_id_i, x_{n_i}-w_id_i}\cup I_{x_{n_i}-w_id_i x_{n_i}}.
$$
and let $Q_i=Q^u_i\cup Q^d_i$.  

Thus $Q_i$ is a rectangle of width $d_i$ and height $2w_id_i$.  Informally $Q^u_i$ and $Q^d_i$ are the upper and lower portions of the rectangle.
\smallskip

Let $B_{Q_i}$ be $\delta_i$-neighborhood of $Q_i$, where
$\{\delta_i\}$ is a fast decaying sequence of small positive numbers such that for $i\ne j$:
	\begin{itemize} 
		\item if $\{m_i, n_i\}\cap \{m_j, n_j\}=\emptyset$, then $B_{Q_i}\cap B_{Q_j}=\emptyset$.
		\item otherwise $B_{Q_i}\cap B_{Q_j}$ is a subset of either $B_{x_{n_i}}$ or $B_{x_{m_i}}$. 
	\end{itemize}
In other words, $B_{Q_i}$ is a small neighborhood of $Q_i$ such that intersection of two such neighborhoods must be inside of  the small neighborhoods of the vertices that were specified above. 
\medskip

\hypertarget{Delta}{\noindent{\bf Defining the open sets $\Delta$.}}
{Let points $x, y\in \mathbb{R}^5$, $x\ne y$}, and a vector $w\in \mathbb{R}^5$ that is not parallel to $(y-x)$ be given.
Let $L(x, y, w)$ be a three dimensional plane that contains $x$ and is orthogonal to
the plane that contains the points $x, y$, and $y+w$, see Figure \ref{f.Delta}.
\begin{definition}\label{d.vfield}
Let $v=v(x, y, w)$ be an affine vector field in $\mathbb{R}^5$ defined by the following properties:
\begin{enumerate}

\item $v|_{L}=0$;

\item $v(y)=w$;

\item $\frac{\partial}{\partial v}(v)=0$ (in other words, non-singular orbits of $v$ are straight lines, and restricted to each of these lines, $v$ is constant).
\end{enumerate}
\end{definition}

\begin{remark}\label{r.psibyv}
\hypertarget{lcpsi}{The vector field $v(x, y, w)$} can be represented as a product of a constant vector field $\bar w$ such that $\bar w(z)=w$ for all $z\in \mathbb{R}^5$ and a $C^\infty$ function $\psi:\mathbb{R}^5\to \mathbb{R}$, $\psi=\psi_{x, y, w}$. We will use this notation later.
\end{remark}

The vector field $v$ with these properties for given $x, y, w$ is uniquely defined. Denote by
$\{g^t_{v}\}$ the flow defined by the vector field $v$.

{Let $U_{x,y,\varepsilon}$ be the cigar defined}  in definition \ref{cigars}. Set 
$\Delta=\Delta(x, y, \varepsilon, w)$ as
$$
\Delta=\cup_{t=0}^1 g_{v}^t(U_{x,y,\varepsilon}).
$$

\begin{figure}[htp]
\centering
 \includegraphics[width=0.7\textwidth]{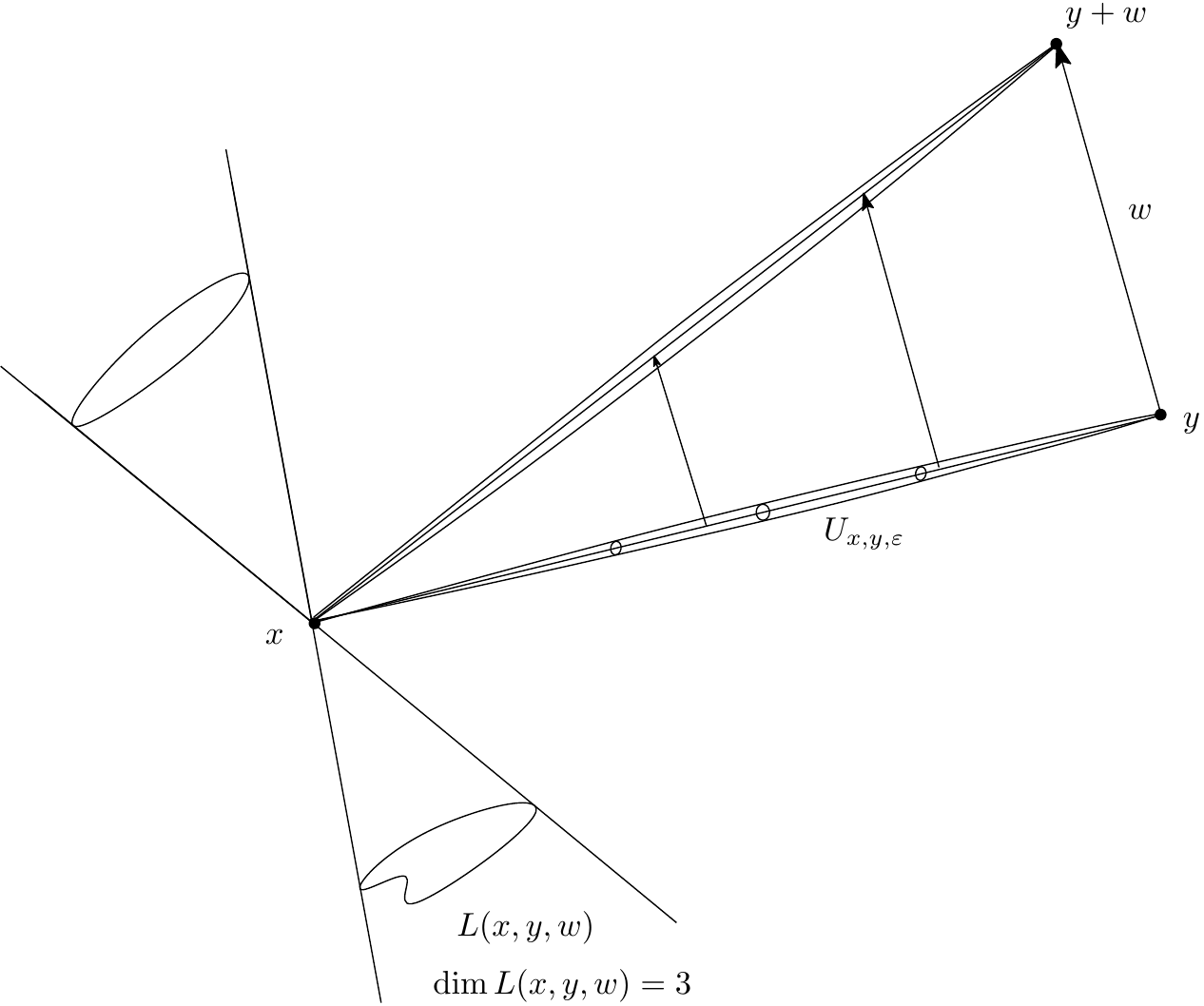}
 \caption{The construction of $\Delta(x, y, w)$.}\label{f.Delta}
\end{figure}

\newpage

\begin{remark}
 \hypertarget{Deltaxyeps}{The set $\Delta=\Delta(x, y, \varepsilon, w)$} is defined to be the union of images of $U_{x,y,\varepsilon}$ when one fixes one end point (in this case $x$), and ``moves" another end point along the vector $w$. In this and other similar definitions (such as the definition of $V(r, x_{n_i}, x_{m_i}, \eta_i, d_iw_i)$ below)
 the first argument
represents the end point that is fixed.
\end{remark}

\paragraph{\bf Defining the open sets $V_{k,i}$ and $W_i$.}
We now define the sets that will contain  the points moved by $h_i$ and are not on cigars connected to the point $r$.

Given points $x, y, z\in \mathbb{E}_3$, $\xi>0$, and a vector $w\in \mathbb{E}_2$, denote
\begin{multline*}
    V(z, x, y, \xi, w) =\Delta(z, x, \xi, w) \cup \Delta(z, x+w, \xi, y-x)\cup \Delta(z, y, \xi, w)\cup \\
    \cup \Delta(z, y, \xi, -w)\cup \Delta(z, y-w, \xi, x-y)\cup \Delta(z, x, \xi, -w).
\end{multline*}

\begin{figure}[htp]
\centering
 \includegraphics[width=0.55\textwidth]{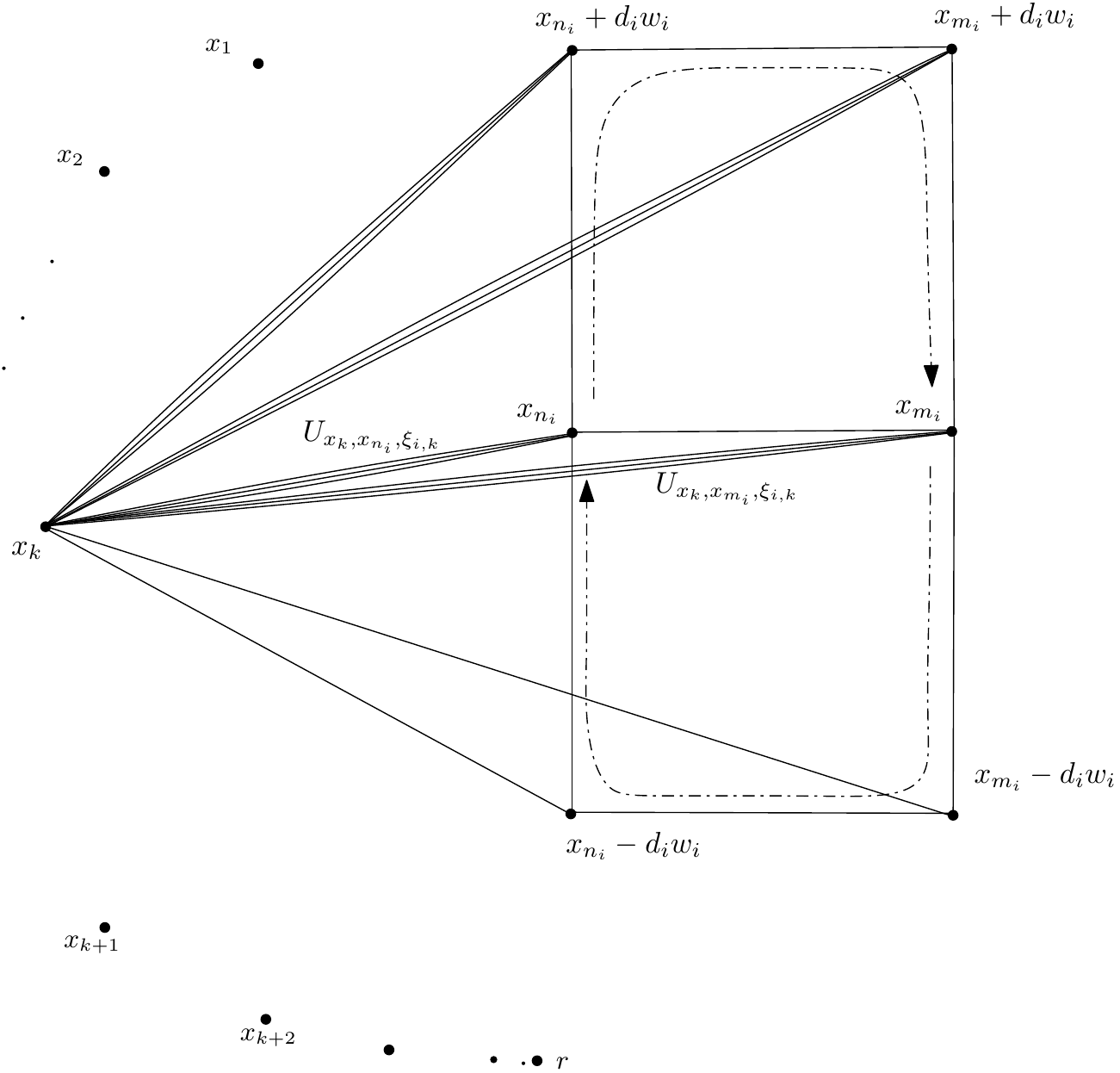}
 \caption{The construction of $V(x_k, x_{n_i}, x_{m_i}, \xi_{i,k}, d_iw_i)$.}
\end{figure}

Also, given points $x, y\in \mathbb{E}_3$,  $\xi>0$, and vector $w\in \mathbb{E}_2$, using definition \ref{d.vfield} we set $\tilde w=v(\frac{x+y}{2}, x, w)$, $\tilde u= v(\frac{x+y}{2}, x+w, y-x)$,  and denote
$$
W(x, y,  \xi, w)=\left(\cup_{t=-1}^1g_{\tilde w}^t(U_{x, y, \xi})\right)\cup \left(\cup_{t=0}^{1}g^t_{\tilde u} (U_{x+w, y-w, \xi})\right)
$$

We  will denote
  $$
  V_{k, i}=V(x_k, x_{n_i}, x_{m_i}, \xi_{i,k},  d_iw_i), \ \ W_i=W(x_{n_i}, x_{m_i}, \xi_i,  d_iw_i),
  $$
where $\xi_{i, k}$, $\xi_i$ are small constants to be chosen later, in Section \ref{ss.absence}.

\subsubsection{\bf Preliminary step: shrinking the cigars.}\label{ss.shrink}  In the construction of the map 
$\mathbf{R}:Graph\to \text{Diff}^\infty\, (\mathbb{R}^5)$ in Proposition \ref{p.reform} the choice of the cigars 
$U_{m,n, \varepsilon_{m,n}}$ was uniform, independent of the graph $E$.  The homeomorphisms $h_i$ are 
built in such a manner that every point in $\mathbb R^5$ is only moved by finitely many $h_i$'s. The 
transformation $h_i$ must swap $m_i$ and $n_i$ and thus move every cigar having either as an endpoint. To 
avoid moving unnecessary points (and points already considered and in the appropriate places) each $h_i$ is a 
composition of maps.  The first of the maps contracts the relevant cigars  diameters. Then the cigars are 
swapped and finally re-expanded. The contraction and expansion preserves the conjugacy classes of the 
relevant maps.

The map used to shrink the cigars will be denoted $S$. 
For any family of positive small constants $\{\varepsilon_{m,n}'\}$, $\varepsilon_{m,n}'< \varepsilon_{m,n}$, we can choose $S$ to be a homeomorphism of $\mathbb{R}^5$ supported on $\overline{\cup_{m\ne n}U_{m,n, \varepsilon_{m,n}}}$ such that $S|_{\mathbf{U}_{m,n}}$ is a contraction towards $\gamma_{m,n}$, and maps \hyperlink{U no epsilon}{$\mathbf{U}_{m,n}$} to $\mathbf{U}_{m,n}'$.  The size of $\epsilon_{m,n}'$ is determined by the transpositions $\phi_i$ built in the decomposition of $\phi$ in Lemma \ref{l.permrepr}.

\subsubsection{\bf Distortion property of homeomorphisms $h_i$.}
It is crucial to be able to make sure that in our construction a point from $U_{x_k, x_{n_i}}$ that was close to $x_k$ and was mapped by $h_i$ to $U_{x_k, x_{m_i}}$ will remain close to $x_k$. In order to give a precise statement and justify it let us introduce some extra notation.

Suppose $U_{x, y, \varepsilon}$ is a \hyperlink{Us}{cigar}. Let us introduce a function on $U_{x, y, \varepsilon}$ that would show whether a given point $z\in U_{x, y, \varepsilon}$ is closer to $x$ or to $y$. Namely, let us recall that an interval $[x, y]$ is parameterized by $t\in [-1, 1]$,
$$
\gamma_{x, y}(t)=\frac{x+y}{2}+t\frac{y-x}{2},
$$
so $\gamma_{x,y}(-1)=x$ and $\gamma_{x,y}(1)=y$. Also, $U_{x,y, \varepsilon}$ was defined as
$$
U_{x,y, \varepsilon} =\bigcup_{-1\le t\le 1}\left\{z\in \Pi_{x, y}(t)\ |\ |z-\gamma_{x,y}(t)|<\varepsilon \psi(t)\right\}
$$
\noindent where $\psi$ is given in equation \ref{psidef}.
\smallskip

Let us define
\begin{equation}\label{e.Pos}
Pos:U_{x,y, \varepsilon}\to [-1,1], \ Pos(z)=t \ \ \text{if}\ \ z\in U_{x,y, \varepsilon}\cap \Pi_{x,y}(t).
\end{equation}
Then if $Pos(z)$ is close to $-1$, then $z$ is close to $x$, and if $Pos(z)$ is close to 1, that $z$ is close to $y$. The function $Pos$ depends on the cigar $U_{x, y , \epsilon}$, and   which of the vertices of the ``cigar'' ($x$ or $y$ in our case) is chosen as first ($Pos(x)=-1$), and which one as a second ($Pos(y)=1$).

\begin{remark}\label{r.35}
The  choice of $\rho_n$ in equation (\ref{e.rho}) implies that for any $x\in B_{x_n}\cap U_{x_n, x_m, \varepsilon_{m,n}}$ we have  $Pos(x)\in [-1, -1+ \frac{1}{1000})$, where $Pos(x)$ is the relative position of $x$ in $U_{x_n, x_m, \varepsilon_{m,n}}$.
\end{remark}
\begin{proposition}\label{p.Pos}
Suppose points $x, y\in \mathbb{R}^5$, and a vector $w\in \mathbb{R}^5, w\ne 0$, $w$ is not collinear to $y-x$, are given. For any $\delta>0$ there exists $\varepsilon_0>0$ such that the following holds. If  $\varepsilon\le \varepsilon_0$, $z_1\in U_{x,y, \varepsilon}$, $z_2\in U_{x,y+w, \varepsilon}$, and the vector $z_2-z_1$ is collinear with $w$, then
$$
|Pos_2(z_2)-Pos_1(z_1)|<\delta
$$
where $Pos_1$ is the position in $U_{x,y, \varepsilon}$ and $Pos_2$ is the position in
$U_{x,y+w, \varepsilon}$.
\end{proposition}

\begin{proof}[Proof of Proposition \ref{p.Pos}.]
Projection of a point in $U_{x, y, \varepsilon}$ to the 2-dimensional plane that contains points $x$, $y$, and $y+w$ does not change the value of the function $Pos$. Therefore, without loss of generality we can assume that both $z_1$ and $z_2$ are the points on that plane.
\begin{figure}[htp]
\centering
 \includegraphics[width=0.7\textwidth]{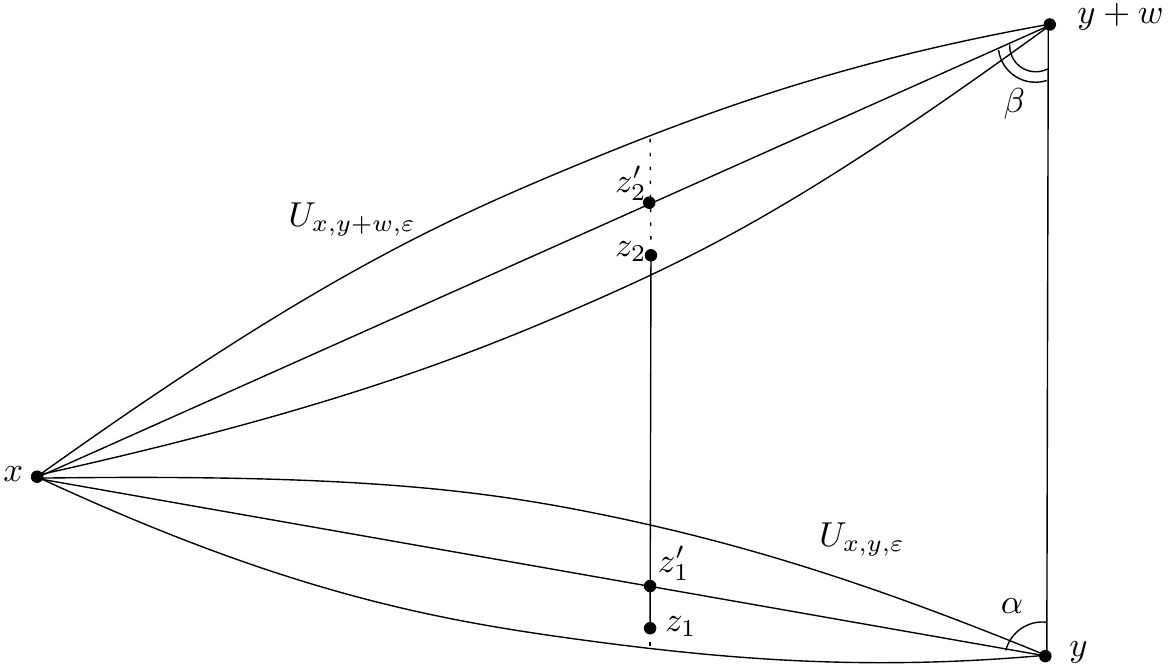}
 \caption{Proof of Proposition \ref{p.Pos}.}
\end{figure}
It is enough to take $$\varepsilon<\frac{1}{4}\min(\text{dist}(x, y), \text{dist}(x, y+w))\min(|\tan\alpha|, |\tan\beta|)\delta,$$
 where $\alpha$ is an angle between an interval $[x,y]$ and the vector $w$, and $\beta$ is an angle between an interval $[x, y+w]$ and the vector $w$. Let $z_1'\in [x, y]$ be a point of the intersection of the interval $[x, y]$ with the line containing $z_1$ and $z_2$, and $z_2'\in [x, y]$ be a point of the intersection of the interval $[x, y+w]$ with that line. Then we have
 $$
 |Pos(z_1)-Pos(z_1')|\le \frac{2}{\text{dist}(x, y)}\frac{\varepsilon}{|\tan \alpha|}\le \frac{\delta}{2},\ \ \text{and}$$
 $$\ \  |Pos(z_2)-Pos(z_2')|\le \frac{2}{\text{dist}(x, y+w)}\frac{\varepsilon}{|\tan \beta|}\le \frac{\delta}{2},
 $$
  and since $Pos(z_1')=Pos(z_2')$, we have
  $$
  |Pos(z_2)-Pos(z_1)|\le  |Pos(z_1)-Pos(z_1')|+|Pos(z_2)-Pos(z_2')|\le \delta.
  $$
\end{proof}

Similarly to Proposition \ref{p.Pos}, the following statement holds:

\begin{proposition}\label{p.PosW}
Suppose points $x, y\in \mathbb{R}^5$, and a vector $w\in \mathbb{R}^5, w\ne 0$, $w$ is not collinear to $y-x$, are given. For any $\delta>0$ there exists $\varepsilon_0>0$ such that the following holds. If  $\varepsilon\le \varepsilon_0$, $z_1\in U_{x,y, \varepsilon}$, $z_2\in U_{x+w,y-w, \varepsilon}$, and the vector $z_2-z_1$ is collinear with $w$, then
$$
|Pos(z_2)-Pos(z_1)|<\delta
$$
\end{proposition}
The proof is very similar to the proof of Proposition \ref{p.Pos}, so we will skip it.

\begin{figure}[htp]
\centering
 \includegraphics[width=0.5\textwidth]{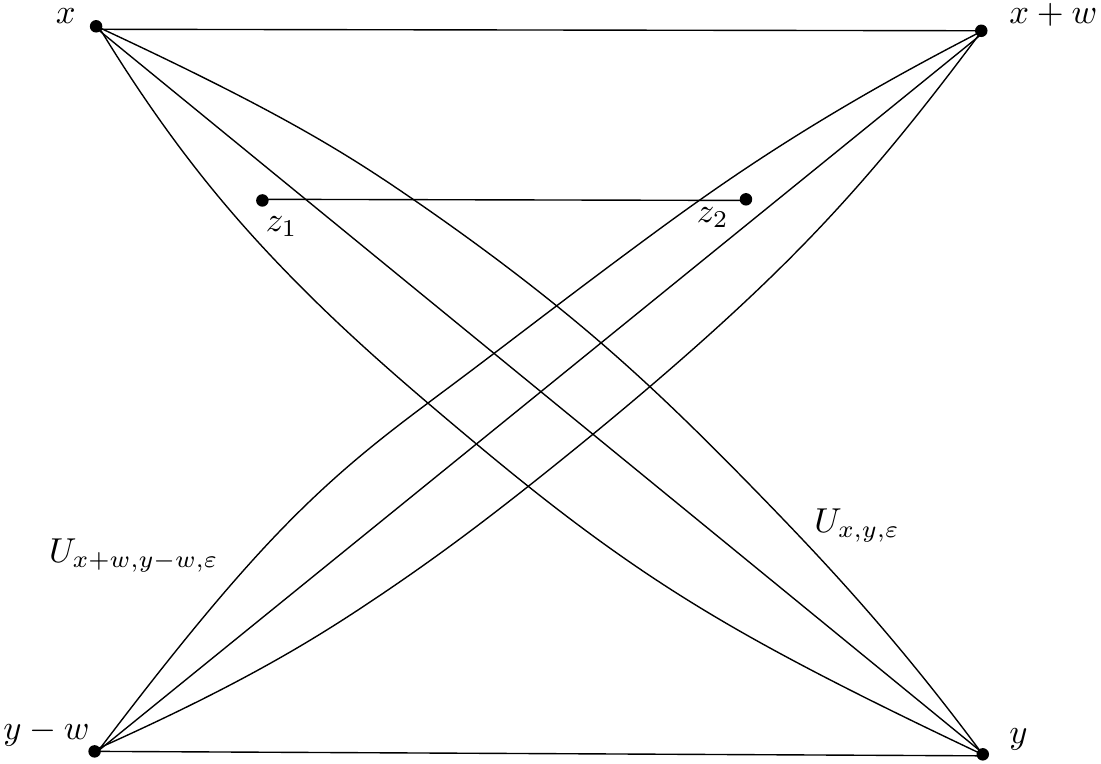}
 \caption{Proposition \ref{p.PosW}.}
\end{figure}

\paragraph{\bf Choice of the sizes of cigars.} As remarked in section \ref{ss.shrink}  the initial sizes of 
cigars can be chosen as small as needed. There are several conditions that we want the choice to satisfy. 
As one of them, we want to make sure that after we constructed a homeomorphism $h_i$ that sends one 
cigar to another, any point from the first cigar that is mapped by $h_i$ to the second cigar does not 
change its position  within a ``cigar'' very much.

Let us start by noticing that Proposition \ref{p.Pos} immediately implies the following statement:
\begin{proposition}\label{p.PosV1to4}
Given points $x, y, t\in \mathbb{E}_3$ that do not belong to the same line, and a vector $w\in \mathbb{E}_2$, for any $\delta>0$ there exists $\varepsilon>0$ such that the following holds.  For any points
$$
z_1\in U_{t, x, \varepsilon}, \ z_2\in U_{t, x+w, \varepsilon}, \ z_3\in U_{t, y+w, \varepsilon}, \ z_4\in U_{t, y, \varepsilon}, \
$$
such that
$$
(z_2-z_1) \parallel w, \ \ \ (z_3-z_2) \parallel (y-x), \ \ \ (z_4-z_3) \parallel (- w), \
$$
we have $$|Pos(z_4)-Pos(z_1)|<\delta,$$ where $Pos(z_4)$ is a position of $z_4$ in $U_{t, y, \varepsilon}$, and $Pos(z_1)$ is a position of $z_1$ in $U_{t, x, \varepsilon}$, as defined in (\ref{e.Pos}).
\end{proposition}
Similarly, Proposition \ref{p.PosW} implies that following statement:
\begin{proposition}\label{p.PosW1to4}
Given distinct points $x, y\in \mathbb{E}_3$ and a vector $w\in \mathbb{E}_2$, for any $\delta>0$ there exists $\varepsilon>0$ such that the following holds.  For any points
$$
z_1\in U_{x, y, \varepsilon}, \ z_2\in U_{x+w, y-w, \varepsilon}, \ z_3\in U_{y+w, x-w, \varepsilon}, \ z_4\in U_{y, x, \varepsilon}, \
$$
such that
$$
(z_2-z_1) \parallel w, \ \ \ (z_3-z_2) \parallel (y-x), \ \ \ (z_4-z_3) \parallel (- w), \
$$
we have $$|Pos(z_4)-Pos(z_1)|<\delta,$$ where $Pos(z_4)$ is a position of $z_4$ in $U_{x, y, \varepsilon}$, and $Pos(z_1)$ is a position of $z_1$ in $U_{y, x, \varepsilon}$, as defined in (\ref{e.Pos}).
\end{proposition}

{Take distinct $n, m$ and $k$ in $\mathbb N$. From item 3 of  Lemma \ref{l.permrepr}, there are at most four values $i\in \mathbb{N}$ such that the transposition $\varphi_i$ moves either $n$ or $m$ (or both). Let us recall that $\rho_n>0$ and $\rho_m>0$ were defined in \hyperlink{rhon}{equation (\ref{e.rho})}. Choose
$\varepsilon_{n,m}>0$ sufficiently small  that if for some $i\in \mathbb{N}$ the transposition $\varphi_i$ moves $n$ and does not move $m$ (or moves $m$ and does not move $n$), then Proposition \ref{p.PosV1to4} is applicable to the points $x=x_n, y=x_k, t=x_m$ (or to the points  $x=x_m, y=x_k,  t=x_n$)} with $\delta=0.001\min(\rho_n, \rho_m)$, and if $\varphi_i$ transposes $n$ and $m$, Proposition \ref{p.PosW1to4} is applicable with the same $\delta>0$.

This choice of the sizes of the cigars will be used in the proof of  to  Corollary \ref{cor.Pos}.

\subsubsection{\bf Absence of intersections.}\label{ss.absence}
  We would like to be able to control the sets of points that are being moved by homeomorphisms $h_i$ in order to justify the property 3) from Proposition \ref{p.key}. This can be done via the following statements:

\begin{proposition}\label{p.support}
  \hypertarget{VkiWi}{There exists a collection of positive numbers $\xi_{i,k}$, $\xi_i$, $i\in \mathbb{N}$, $k\in \mathbb{N}$,} such that the following holds.
Recall the notation:
  $$
  V_{k, i}=V(x_k, x_{n_i}, x_{m_i}, \xi_{i,k},  d_iw_i), \ \ W_i=W(x_{n_i}, x_{m_i}, \xi_i,  d_iw_i),
  $$
  then for $i\ne j$: 
  	\begin{enumerate}
   \item $V_{k, i}\cap V_{l,j}\ne \emptyset$ if and only if one of the pairs $\{n_i, k\}$, $\{m_i, k\}$ coincides with one of the pairs $\{n_j, l\}$, $\{m_j, l\}$, 
    \item $W_i\cap W_j=\emptyset$ if $i\ne j$,
  \item $W_i\cap V_{k, j}\ne \emptyset$ if and only if $\{n_i, m_i\}=\{n_j, k\}$ or $\{n_i, m_i\}=\{m_j, k\}$.
  
   \end{enumerate}
\end{proposition}
\begin{proof}[Proof of Proposition \ref{p.support}]
Let us denote by $\pi_2:\mathbb{R}^5\to \mathbb{E}_2$ the projection to the two-dimensional subspace, and by $\pi_3:\mathbb{R}^5\to \mathbb{E}_3$ the projection to the three dimensional subspace. We first consider the case $k\ne l$. For given $i\in \mathbb{N}$, consider the plane $\mathcal{P}_{i,k}$ that contains $x_{n_i}, x_{m_i}$, and $x_k$.
 The family  of planes $\mathcal{P}_{i,k}$ converges to the plane that contains points $x_{n_i}, x_{m_i}$, and \hyperlink{defofr}{$r$}  as $k\to \infty$.  Denote the limit by $\mathcal{P}_{i,\infty}$.  It follows that for each $k$ the infimum of the angles between $\mathcal{P}_{i,k}$ and any $\mathcal{P}_{i,k'}$, $k'\ne k$, is positive. This implies that if $\{\xi_{i,k}\}_{k\in \mathbb{N}}$ are sufficiently small, then for any $k\ne k'$ we have $\pi_3(V_{k, i})\cap \pi_3(V_{k', i})=\emptyset$, hence $V_{k, i}\cap V_{k', i}=\emptyset$.

If $k=l$ and none of the pairs coincide, then $\{n_i, m_i\}\cap \{n_j, m_j\}=\emptyset$. For $\xi_{i, k}, \xi_{j, l}$ small enough,  we have $V_{k,i}\cap V_{l, j}=\emptyset$. Indeed, consider $U_{x_k,  x_{n_i}, \varepsilon_{n_i, k}}\cup B_{x_{n_i}}$. There exists $\delta>0$ such that if $\xi_{i,k}$ is sufficiently small, then for any $t\in [-\delta, \delta]$ we have
$$
g^t_{v}(U_{x_k, x_{n_i}, \xi_{i,k}}) \subset U_{x_k, x_{n_i},  \varepsilon_{n_i, k}}\cup B_{k},
$$
where the vector field $v=v(x_k, x_{n_i}, d_iw_i)$ is given by definition \ref{d.vfield}, and for any 
$t\in [-1, -\delta]\cup [\delta, 1]$ we have that
$
\pi_2(g^{t}_{v}(U_{x_k, x_{n_i}, \xi_{i,k}}))
$
is in the \hyperlink{conesK}{cone $K_i$} (the cone in $\mathbb{E}_2$  of size $\frac{\beta_i}{3}$ around $w_i$, given in  definition \ref{d.coneK} above). Similarly, we will have for any $t\in [-\delta, \delta]$
$$
g^t_{v}(U_{x_k, x_{m_i}, \xi_{i,k}}) \subset U_{x_k, x_{m_i},  \varepsilon_{m_i, k}}\cup B_{x_{m_i}},
$$
where $v=v(x_k, x_{m_i}, d_iw_i)$, and for any $\tau\in [-1, -\delta]\cup [\delta, 1]$ we will have that
$
\pi_2(g^{\tau}_{v}(U_{x_k, x_{m_i}, \xi_{i,k}}))
$
is in the cone $K_i$. 
Finally, by choosing $\xi_{i,k}$ sufficiently small, we can guarantee that for any $s\in [0,1]$ the projection
$
\pi_2(g^{s}_{v}(U_{x_k, x_{n_i}+d_i w_i, \xi_{i,k}})),
$
 where $v=v(x_k, x_{n_i}+d_iw_i, x_{m_i}-x_{n_i})$, is in the cone $K_i$. 
  Since the cones $K_i$ and $K_j$ 
  are disjoint if $i\ne j$, this implies that if $\{n_i, m_i\}\cap \{n_j, m_j\}=\emptyset$, then 
  $V_{k, i}\cap V_{l, j}=\emptyset$ for any $k\in \mathbb{N}\backslash \{n_i, m_i\}$ and any 
  $l\in \mathbb{N}\backslash \{n_j, m_j\}$.
This proves the first item.

 Similar arguments show that if $\{\xi_n\}$ are positive are sufficiently small and go to zero fast enought, then $W_i\cap W_j=\emptyset$ if $i\ne j$.   Indeed, 
  if $\xi_i>0$ are small, then  projections of $W_i$ and $W_j$ to $\mathbb{E}_3$ must be disjoint. Hence the second item holds.

 What is left is to show that
 $W_i\cap V_{k, j}= \emptyset$ if  $\{n_i, m_i\}\ne\{n_j, k\}$ and $\{n_i, m_i\}\ne\{m_j, k\}$. Notice that if $\xi_i$ is small, then $\pi_3(W_i)\subseteq U_{x_{n_i}, x_{m_i}, \varepsilon_{n_i, m_i}}$. In particular, for any $\{n_i, m_i\}\ne\{n_j, k\}$ and $\{n_i, m_i\}\ne\{m_j, k\}$ we have
 \begin{eqnarray}
 \left(\cup_{t=-1}^{1}g^t_{v(x_k, x_{n_j}, d_jw_j)}(U_{x_k, x_{n_j}, \varepsilon_{k, n_j}})\right)\cap&W_i=\emptyset, \notag \\
 \mbox{and}&\notag\\
  \left(\cup_{t=-1}^{1}g^t_{v(x_k, x_{m_j}, d_jw_j)}(U_{x_k, x_{m_j}, \varepsilon_{k, m_j}})\right)\cap&W_i=\emptyset. \notag
\end{eqnarray}
Now, if $\xi_i$ is sufficiently small, for any point $p\in W_i\backslash U_{x_{n_i}, x_{m_i}, \varepsilon_{n_i, m_i}}$ the projection $\pi_2(p)$ is in the cone $K_i$.  Therefore, for any $j\ne i$ we have
$$
\left(\cup_{t=0}^1g^t_{v(x_k, x_{n_j}+d_jw_j, x_{m_j}-x_{n_j})}(U_{x_k, x_{n_j}+d_jw_j, 
 \xi_j })\right)\cap W_i=\emptyset,
$$
and
$$
\left(\cup_{t=0}^1g^t_{v(x_k, x_{m_j}-d_jw_j,  x_{n_j}-x_{m_j})}(U_{x_k, 
x_{m_j}-d_jw_j, \xi_j })\right)\cap W_i=\emptyset.
$$
This implies that $V_{k, j}\cap W_i = \emptyset$.

Finally, since the angles between planes containing the points $x_{n_i}$, $x_{m_i}$, and  $x_{k}$ and the plane containing points $x_{n_i}$, $x_{m_i}$, and  $x_{n_i}+d_iw_i$ are non-zero, if $\xi_i$ and $\xi_{i,k}$ are sufficiently small, we also have
$$
V(x_k, x_{n_i}, x_{m_i}, \xi_{i, k}, d_iw_i)\cap W_i=\emptyset.
$$
Proof of Proposition \ref{p.support} is complete.
\end{proof}
We also would like to choose the sizes $\{\xi_{i,k}\}$ and $\{\xi_i\}$ so small that after we construct the homeomorphisms $\{h_i\}$, a map $h_i$ will not change a relative position $Pos$ of a point too much. Formally, it can be done by the following statement (that follows immediately from Propositions \ref{p.PosV1to4} and \ref{p.PosW1to4}):
\begin{proposition}\label{p.xiPos}
  There exists a collection of positive numbers $\xi_{i,k}$, $\xi_i$, $i\in \mathbb{N}$, $k\in \mathbb{N}$, such that additionally to the conditions stated in Proposition \ref{p.support} the following holds:

  For any points
$$
z_1\in U_{x_k, x_{n_i}, \xi_{i,k}}, \ z_2\in U_{x_k, x_{n_i}+d_iw_i, \xi_{i,k}}, \ z_3\in U_{x_k, x_{m_i}+d_iw_i, \xi_{i,k}}, \ z_4\in U_{x_k, x_{m_i}, \xi_{i,k}}, \
$$
such that
$$
(z_2-z_1) \parallel w_i, \ \ \ (z_3-z_2) \parallel (x_{m_i}-x_{n_i}), \ \ \ (z_4-z_3) \parallel (- w_i), \
$$
we have $$|Pos(z_4)-Pos(z_1)|<\frac{1}{500\cdot 2^i},$$ where $Pos(z_4)$ is a position of $z_4$ in $U_{x_k, x_{m_i}, \xi_{i,k}}$, and $Pos(z_1)$ is a position of $z_1$ in $U_{x_k, x_{n_i}, \xi_{i,k}}$, as defined in (\ref{e.Pos}).

   Also, for any points
$$
z_1\in U_{x_{n_i}, x_{m_i}, \xi_{i}}, \ z_2\in U_{x_{n_i}+d_iw_i, x_{m_i}-d_iw_i, \xi_i}, \ z_3\in U_{x_{n_i}+d_iw_i, x_{m_i}-d_iw_i, \xi_i}, \ z_4\in U_{x_{m_i}, x_{n_i}, \xi_i}, \
$$
such that
$$
(z_2-z_1) \parallel w_i, \ \ \ (z_3-z_2) \parallel (x_{m_i}-x_{n_i}), \ \ \ (z_4-z_3) \parallel (- w_i), \
$$
we have $$|Pos(z_4)-Pos(z_1)|<\frac{1}{500\cdot 2^i},$$ where $Pos(z_4)$ is a position of $z_4$ in $U_{x_{m_i}, x_{n_i}, \xi_i}$, and $Pos(z_1)$ is a position of $z_1$ in $U_{x_{n_i}, x_{m_i}, \xi_i}$, as defined in (\ref{e.Pos}).
\end{proposition}

Now we would like to choose the sizes of the initial cigars so small that after we construct a homeomorphism $h_i$, their images will still  be contained in the cigars of sizes dictated by Propositions  \ref{p.support}  and \ref{p.xiPos}. Let us start with the following simple statement:

\begin{proposition}
  Suppose a collection of positive numbers $\xi_{i,k}$, $\xi_i$, $i\in \mathbb{N}$, $k\in \mathbb{N}$, satisfies the conditions stated in Propositions \ref{p.support} and \ref{p.xiPos}. Then there exists a collection of positive numbers $\xi_{i,k}'$, $\xi_i'$, $i\in \mathbb{N}$, $k\in \mathbb{N}$, such that the following inclusions hold for all $i\in \mathbb{N}$ and $k\ne n_i, m_i$:
  $$
  g^1_{v(x_k, x_{n_i}, d_jw_j)}(U_{x_k, x_{n_i}, \xi_{i, k}'})\subset U_{x_k, x_{n_i}+d_iw_i, \frac{1}{2}\xi_{i, k}},
  $$
  $$
  g^1_{v(x_k,  x_{n_i}+d_iw_i, x_{m_i}-x_{n_i})}\circ g^1_{v(x_k, x_{n_i}, d_iw_i)}(U_{x_k, x_{n_i}, \xi_{i, k}'})\subset U_{x_k, x_{m_i}+d_iw_i, \frac{1}{2}\xi_{i, k}},
  $$
 $$
 g^1_{v(x_k, x_{n_i}+d_iw_i, -d_iw_i)} \circ g^1_{v(x_k,  x_{n_i}+d_iw_i, x_{m_i}-x_{n_i})}\circ g^1_{v(x_k, x_{n_i}, d_iw_i)}(U_{x_k, x_{n_i}, \xi_{i, k}'})\subset U_{x_k, x_{m_i}, \frac{1}{2}\xi_{i, k}},
  $$
  $$
  g^1_{v(x_k, x_{m_i}, -d_jw_j)}(U_{x_k, x_{m_i}, \xi_{i, k}'})\subset U_{x_k, x_{m_i}-d_iw_i, \frac{1}{2}\xi_{i, k}},
  $$
  $$
  g^1_{v(x_k,  x_{m_i}-d_iw_i, x_{n_i}-x_{m_i})}\circ g^1_{v(x_k, x_{m_i}, -d_iw_i)}(U_{x_k, x_{m_i}, \xi_{i, k}'})\subset U_{x_k, x_{n_i}-d_iw_i, \frac{1}{2}\xi_{i, k}},
  $$
 $$
 g^1_{v(x_k, x_{n_i}-d_iw_i, d_iw_i)} \circ g^1_{v(x_k,  x_{m_i}-d_iw_i, x_{n_i}-x_{m_i})}\circ g^1_{v(x_k, x_{m_i}, -d_iw_i)}(U_{x_k, x_{m_i}, \xi_{i, k}'})\subset U_{x_k, x_{n_i}, \frac{1}{2}\xi_{i, k}},
  $$
   $$
  g^1_{v(\frac{1}{2}(x_{n_i}+x_{m_i}), x_{n_i}, d_jw_j)}(U_{x_{n_i}, x_{m_i}, \xi_{i}'})\subset U_{x_{n_i}+d_iw_i, x_{m_i}-d_iw_i, \frac{1}{2}\xi_{i}},
  $$
 $$
  g^1_{v(\frac{1}{2}(x_{n_i}+x_{m_i}), x_{n_i}+d_iw_i, x_{m_i}-x_{n_i})}\circ\\
   g^1_{v(\frac{1}{2}(x_{n_i}+x_{m_i}), x_{n_i}, d_jw_j)}(U_{x_{n_i}, x_{m_i}, \xi_{i}'})\subset U_{x_{m_i}+d_iw_i, , x_{n_i}-d_iw_i, \frac{1}{2}\xi_{i}},
  $$

\begin{eqnarray*}
 g^1_{v(\frac{1}{2}(x_{n_i}+x_{m_i}), x_{m_i}+d_iw_i, -d_iw_i)} \circ g^1_{v(\frac{1}{2}(x_{n_i}+x_{m_i}), x_{n_i}+d_iw_i, x_{m_i}-x_{n_i})}\circ\\
  g^1_{v(\frac{1}{2}(x_{n_i}+x_{m_i}), x_{n_i}, d_jw_j)}(U_{x_{n_i}, x_{m_i}, \xi_{i}'})\subset U_{x_{m_i}, x_{n_i}, \frac{1}{2}\xi_{i}}.
\end{eqnarray*}
\end{proposition}
Notice that the same pair of indices of vertices can serve as a pair $(n_i, k)$ for some $i\in \mathbb{N}$, and at the same time, as a pair $(n_j, m _j)$ for some $j\ne i$, which gives several restrictions on the size of an initial ``cigar'' with vertices $x_n$ and $x_m$. So we will set
$$
  \varepsilon_{m,n}'=\min\{\xi_{i,k}', \xi_i'\ |\ \{m,n\}=\{n_i, k\}\ \text{or}\ \{m,n\}=\{m_i, k\} \ \text{or}\  \{m,n\}=\{n_i, m_i\}\ \}
$$
Because any $n$ is only moved by finitely many $\{\varphi_i\}$ 
the minimum above is taken over a finite set. Now we can apply the arguments from Section \ref{ss.shrink} and start with the initial cigars of sizes $\{\varepsilon_{m,n}'\}$.

\subsubsection{\bf Coating the needles}

Here we choose a sequence of cigars $U_{r, x_n, \tilde \theta_n}$ around the ``needles'' of $\mathbf{I}$. After that we will make the sizes of those cigars even smaller, so we reserve the notation $\{\theta_n\}$ for the preliminary sizes of those neighborhoods, and will use $\{\tilde \theta_n\}$ for the adjusted sizes.

\begin{lemma}\label{l.O3}
There exists a sequence $\{\tilde\theta_n\}_{n\in \mathbb{N}}$ of positive numbers such the following hold: \\
 1) For any $m\ne n$ one has $U_{r, x_n, \tilde\theta_n} \cap U_{r, x_m, \tilde\theta_m} =\emptyset$ \\
 2) 
 For any distinct $m, n, k$ that are all different we have 
 \[U_{r, x_n, \tilde\theta_n} \cap U_{x_m, x_k, \varepsilon_{m,k}}=\emptyset\]
 3) If $\{n_i, m_i\}\cap \{n_j, m_j\}=\emptyset$, then
 $$
 V(r, x_{n_i}, x_{m_i}, \tilde\theta_{n_i}, d_iw_i) \cap V(r, x_{n_j}, x_{m_j}, \tilde\theta_{n_j}, d_jw_j) =\emptyset,
 $$
 and for any $k\in \mathbb{N}$
  $$
 V(r, x_{n_i}, x_{m_i}, \tilde\theta_{n_i}, d_iw_i) \cap V(x_k, x_{n_j}, x_{m_j}, \xi_{j, k}, d_jw_j) =\emptyset
 $$
 4) For any $i, j\in \mathbb{N}$ (including the case when $j=i$) we have $$W_j\cap V(r, x_{n_i}, x_{m_i}, \tilde\theta_{n_i}, d_iw_i)=\emptyset$$
\end{lemma}
\begin{proof}[Proof of Lemma \ref{l.O3}]
Let us recall that the vertices $\{x_n\}\subset \mathbb{R}^5$ were chosen in such a way that properties listed in Proposition \ref{p.1} are satisfied. In particular, due to property 6) from Proposition \ref{p.1}, it is enough to choose some $\theta'_n<\eta_n$ for each $n\in \mathbb{N}$ to make sure that conditions 1) and 2) %
 from Lemma \ref{l.O3} are satisfied for any $\tilde\theta_n\le \theta_n'$.

The arguments to show that for sufficiently small $\tilde \theta_n$ the properties 3) and 4) hold are parallel to the arguments used in the proof of Lemma \ref{p.support}, so we will omit them here.
\end{proof}

Let us now explain how we are going to choose the actual sizes $\{\theta_n\}$ of the neighborhoods $U_{r, x_n, \theta_n}$. Recall from  definition \ref{d.JTn}:
$$
J(n)=\min\{j\in \mathbb{N} \ |\ \text{ for all }  m\in Cont_n, \ m \ \text{is frozen at stage $j$} \}
$$
and
$$
T(n)=\max\{n_i, m_i \ |\ i\le J(n)\},
$$
where $n_i, m_i$ are the elements of $\mathbb{N}$ that are swapped by $\varphi_i$.

For each $n\in \mathbb{N}$ choose small $\theta_n>0$ such that $\theta_n\le \tilde \theta_n$, and $U_{x_n, r, \theta_n}\cap U_{x_n, x_m, \varepsilon_{n,m}}=\emptyset$ for all $m\le T(n)$.
\begin{remark}\label{r.mnmoving}
 Suppose $U_{x_n, r, \theta_n}\cap U_{x_n, x_k, \varepsilon_{n,k}}\ne\emptyset$. Then:
 \begin{enumerate}
 \item $k\ge T(n)$.
\item If $h_i$ moves $x_k$, then for any $j\ge i$ the map $h_j$ does not move $x_n$ (in other words, $n$ is frozen at stage $i$). 
\end{enumerate}
\end{remark}

Let us define the sets $Base_i\subset \mathbb{R}^5$ as
\begin{equation}\label{e.Basei}
Base_i=\left(\bigcup_{n\in \mathbb{N}}B_{x_n}\right)\bigcup \left(\bigcup_{n\ne m}U_{x_n, x_m, \varepsilon_{n,m}'}\right)\bigcup \left(\bigcup_{n \text{ is not frozen at stage $i$}}U_{x_n, r, \theta_n}\right)
\end{equation}

\subsubsection{\bf Construction of homeomorphisms $\{h_i\}$.}
Here we present the construction of the homeomorphisms $\{h_i\}$. Fix an $i$ for the rest of this subsection.
\smallskip

We will need the following simple statement:
\begin{lemma}\label{l.contpsi}
Suppose $U\subset \mathbb{R}^d$ is bounded open set, and $\mathfrak{K}\subset \bar U$ is compact. Suppose $\psi:\mathfrak{K}\to \mathbb{R}$ is a continuous function such that $\psi(x)=0$ for any point $x\in \mathfrak{K}\cap \partial U$. Then there exists a continuous function $\Psi:\bar U\to \mathbb{R}$ such that $\Psi|_{\mathfrak{K}}=\psi$, and $\Psi|_{\partial U}=0$.
\end{lemma}
We will construct $h_i$ as a composition of six maps,
$$
h_i=R_i\circ \hat h_i\circ h^{\uparrow}_i\circ h^{\rightarrow}_i\circ h_i^{\uparrow}\circ S_i.$$

The homeomorphism $S_i$ will shrink the sizes of the cigars connecting the $x_k$'s and $r$ to the points $x_{m_i}$ and $x_{n_i}$ and the cigar connecting $x_{m_i}$ to $x_{n_i}$. This is done to make them small enough that moving them along the paths determined by the superscripted $h_i$'s keeps their intersections with the other cigars very small.   The final homeomorphism $R_i$ will increase the sizes back to match the sizes of the target cigars.

Let us describe each of them separately.

\vspace{4pt}

\hypertarget{Si}{\noindent{\bf Definition of $S_i$.}}\\

The homeomorphism $S_i$ will be supported on
\begin{equation}\label{e.union}
\left(\bigcup_{k\ne n_i, m_i}\left(U_{x_k, x_{n_i}, \varepsilon_{k, n_i}}\cup U_{x_k, x_{m_i}, \varepsilon_{k, m_i}}\right)\right)\cup U_{x_{n_i}, x_{m_i}, \varepsilon_{n_i, m_i}}
\end{equation}
All sets in this union are pairwise disjoint, so one can define $S_i$ on each of them separately. For example, construct $S_i$ on $U_{x_k, x_{n_i}, \varepsilon_{k, n_i}}$. This will depend on a small positive parameter $\delta$ which in turn depends on $k$ and $i$. The size of $\delta$ will be specified later. 
Let $S_i$ restricted to 
\hyperlink{U no epsilon}{$\mathbf{U}_{x_k, x_{n_i}}=U_{x_k, x_{n_i}, \frac{1}{2}\varepsilon_{k, n_i}}$} 
be a contraction toward $\gamma_{x_k, x_{n_i}}$ (the line that connects $x_{k}$ and $x_{n_i}$) with coefficient 
$\frac{2\delta}{\varepsilon_{k, n_i}}$, so that $S_i(\mathbf{U}_{x_k, x_{n_i}})=U_{x_k, x_{n_i}, \delta}$. 
Also, we will require that $S|_{\partial U_{x_k, x_{n_i}, \varepsilon_{k, n_i}}}=id$, and for any ray 
$\mathfrak{r}$ that originates at a point of $\gamma_{x_k, x_{n_i}}$ and is orthogonal to it, $S_i$ is defined on $\mathfrak{r}\cap \left(U_{x_k, x_{n_i}, \varepsilon_{k, n_i}}\backslash \mathbf{U}_{x_k, x_{n_i}}\right)$ as an affine map that sends 
it to $\mathfrak{r}\cap \left(U_{x_k, x_{n_i}, \varepsilon_{k, n_i}}\backslash U_{x_k, x_{n_i}, \delta}\right)$. Notice that these properties define $S_i$ on $U_{x_k, x_{n_i}, \varepsilon_{k, n_i}}$ uniquely.

\vspace{4pt}

\noindent{\bf Definition of $h^{\uparrow}_i$.}\\
Let us apply Lemma \ref{l.contpsi} to an open set $U$ given by the union
\begin{equation}\label{e.U}
U=B_{Q_i}\cup W_i\cup \left(\cup_{k\ne n_i, m_i}V_{k, i}\right)\cup  V(r, x_{n_i}, x_{m_i}, \theta_i, d_iw_i)
\end{equation}
\hyperlink{Deltaxyeps}{Define}
	\begin{eqnarray*}
	\mathfrak{K}=&\overline{\bigcup_{k\ne n_i, m_i}\left({\Delta(x_k, x_{n_i}, \frac{1}{2}\xi_{k, i}, \pm d_iw_i)}	
	\cup \Delta(x_k, x_{m_i}, \frac{1}{2}\xi_{k, i}, \pm d_iw_i) \right)}
	\\
 	& \cup\ \overline{W(x_{n_i}, x_{m_i}, \frac{1}{2}\xi_i, d_iw_i)}
	\end{eqnarray*}
Notice that $\mathfrak{K}\subset \bar U$, and $\mathfrak{K}\cap \partial U=\{x_k\}_{k\ne n_i, m_i }\cup \{r\}$. 
Now  define $\psi:\mathfrak{K}\to \mathbb{R}$ in the following way:

\begin{itemize}
	\item If $x\in \Delta(x_k, x_{n_i}, \frac{1}{2}\xi_{k, i}, d_iw_i)\cup \Delta(x_k, x_{n_i}, \frac{1}{2}\xi_{k, i}, -d_iw_i)$, then $\psi$ is given by Remark \ref{r.psibyv} applied to $v=v(x_k, x_{n_i}, d_iw_i)$.
	\item If $x\in \Delta(x_k, x_{m_i}, \frac{1}{2}\xi_{k, i}, -d_iw_i)\cup \Delta(x_k, x_{m_i}, \frac{1}{2}\xi_{k, i}, d_iw_i)$, then $\psi$ is given by Remark \ref{r.psibyv} applied to $v=v(x_k, x_{m_i}, -d_iw_i)$.
	\item If $x\in \overline{W(x_{n_i}, x_{m_i}, \frac{1}{2}\xi_i, d_iw_i)}$, then $\psi$ is given by Remark \ref{r.psibyv} applied to $v=v(\frac{x_{n_i}+x_{m_i}}{2}, x_{n_i}, d_iw_i)$.
\end{itemize}

By continuity, $\psi$ is defined on the whole set $\mathfrak{K}$, including the needles 
$U_{r, x_{m_i}, \tilde{\theta}_{m_i}}$ and $U_{r, x_{n_i}, \tilde{\theta}_{n_i}}$. 
Since $\mathfrak{K}\cap \partial U=\{x_k\}_{k\ne n_i, m_i}\cup\{r\}$, it is clear that $\psi |_{\mathfrak{K}\cap \partial U}=0$. Moreover using the techniques involving the function \hyperlink{flatfixed}{$\Theta$} (Item \ref{propsoftheta}) used in Definition \ref{plusnminus}, we can assume that every element of $\mathbf{I}$ is a flat fixed point of $\psi$.

So we can apply Lemma \ref{l.contpsi} to construct a continuous function $\Psi:\mathbb{R}^5\to \mathbb{R}$, such that $\Psi|_{\mathfrak{K}}=\psi$ and $\Psi|_{\mathbb{R}^5\backslash U}=0$. 

Define $h^{\uparrow}_i$ to be  a time-one shift along the vector filed $\Psi\cdot \bar w_i$, where $\bar w_i$ is a constant vector field in the direction of $w_i$.

\vspace{4pt}

\noindent{\bf Definition of $h^{\rightarrow}_i$.}\\
 We will use a similar strategy. First, define $U$ in the same way as in (\ref{e.U}).
Now define
\begin{align*}
\mathfrak{K}=&\overline{\bigcup_{k\ne n_i, m_i}\Delta(x_k, x_{n_i}+w_i, \frac{1}{2}\xi_{k,i}, x_{m_i}-x_{n_i})}\cup \overline{\bigcup_{k\ne n_i, m_i}\Delta(x_k, x_{m_i}-w_i, \frac{1}{2}\xi_{k,i}, x_{n_i}-x_{m_i})}\cup\\
 &\overline{W(x_{n_i}, x_{m_i}, \frac{1}{2}\xi_i, d_iw_i)}
\end{align*}
Notice that $\mathfrak{K}\subset \bar U$, and 
	\begin{equation*}
	\mathfrak{K}\cap \partial U=\{x_k\}_{k\ne n_i, m_i }\cup \{r\}
	\end{equation*}
Let us define a function $\psi:\mathfrak{K}\to \mathbb{R}$ in the following way:
	\begin{itemize}
	\item If $x\in \Delta(x_k, x_{n_i}+w_i, \frac{1}{2}\xi_{k,i}, x_{m_i}-x_{n_i})$, then $\psi$ is given by 
	Remark \ref{r.psibyv} applied to $v=v(x_k, x_{n_i}+d_iw_i, x_{m_i}-x_{n_i})$.\\
	\item If $x\in \Delta(x_k, x_{m_i}-w_i, \frac{1}{2}\xi_{k,i}, x_{n_i}-x_{m_i})$, then  $\psi$ is given by 
	Remark \ref{r.psibyv} applied to $v=v(x_k, x_{n_i}-d_iw_i, x_{n_i}-x_{m_i})$.\\
	\item If $x\in \overline{W(x_{n_i}, x_{m_i}, \frac{1}{2}\xi_i, d_iw_i)}$, then $\psi$ is given by 
	Remark \ref{r.psibyv} applied to $v=v(\frac{x_{n_i}+x_{m_i}}{2}, x_{n_i}+d_iw_i, x_{m_i}-x_{n_i})$.
	\end{itemize}

Exactly in the same way as in definition of $h_i^{\uparrow}$, by continuity $\psi$ is defined on the whole set $\mathfrak{K}$. Since $\mathfrak{K}\cap \partial U=\{x_k\}_{k\ne n_i, m_i}\cup\{r\}$, it is clear that $\psi|_{\mathfrak{K}\cap \partial U}=0$. So we can apply Lemma \ref{l.contpsi} to construct a continuous function $\Psi:\mathbb{R}^5\to \mathbb{R}$, such that $\Psi|_{\mathfrak{K}}=\psi$ and $\Psi|_{\mathbb{R}^5\backslash U}=0$.

Define $h^{\rightarrow}_i$ to be a time-one shift along the vector filed $\Psi\cdot v$, where $v$ is a constant vector field in the direction of $x_{m_i}-x_{n_i}$.

\

\noindent{\bf Definition of $\hat h_i$.}\\
To define $\hat h_i$, we will need the following statement:
\begin{lemma}\label{l.lcon}
Let $L:\mathbb{R}^5\to \mathbb{R}^5$ be an affine map that sends an interval $I=[x, y]$ into the interval $I'=[x', y']$. Suppose open sets $U_{x_1, x_2, \varepsilon}$ and  $U_{x_1', x_2', \varepsilon'}$ for some $\varepsilon, \varepsilon'>0$ are two cigars given via definition \ref{cigars}. Fix a small $\mu>0$.  Then for any sufficiently small $\delta, \delta'>0$ the following holds.

Let $\Xi$ (respectively, $\Xi'$) be a real valued $C^\infty$ function  that it is positive in   $U_{x, y, \varepsilon}$ (respectively, in $U_{x', y', \varepsilon'}$) and vanishes outside of it. Let $w$ (respectively, $w'$) be a vector field orthogonal to the line $\gamma_{x, y}$ containing points $x$ and $y$  (respectively, to the line $\gamma_{x', y'}$), directed towards this line, and such that $\|w(z)\|=\text{dist}(z, \gamma_{x, y})$ (respectively, $\|w'(z)\|=\text{dist}(z, \gamma_{x', y'})$). Define  $f$ (respectively $f'$) to be the time-1 shift along the flow defined by the vector field $\Xi\cdot w$ (respectively, $\Xi'\cdot w'$).

Define a map $S^*:U_{x, y, \varepsilon}\to U_{x, y, \delta}$ as a contraction toward the line $\gamma_{x, y}$ with coefficient $\frac{\delta}{\varepsilon}<1$. Similarly, define a map $R^*:U_{x', y', \delta'}\to U_{x', y', \varepsilon'}$ as an inverse of a contraction toward the line $\gamma_{x', y'}$ with coefficient $\frac{\delta'}{\varepsilon'}<1$. 
There exists a homeomorphism
$$
\hat h:\overline{U_{x', y', \varepsilon'}}\to \overline{U_{x', y', \varepsilon'}}, \ 
$$
 such that

\vspace{3pt}
\begin{enumerate}
	\item $\hat h|_{\partial U_{x', y', \varepsilon}}=\text{id}$,
	\item $\hat h(L(U_{x, y, \delta}))=U_{x', y', \delta'}$,
	\item   if $\tilde h=R^*\circ \hat h\circ L\circ S^*$, then $f'\circ \tilde h=\tilde h\circ f$,
	
	\item $\text{dist}_{C^0}(id, \hat h)<\mu$.
\end{enumerate}
\end{lemma}
\smallskip

\noindent Notice that $S^*=S^*(x, y, \varepsilon, \delta)$ is a function of $x, y, \varepsilon$ and $\delta$,  and 
similarly $R^*=R^*(x', y', \varepsilon', \delta')$ is a function of $x', y', \varepsilon'$ and $\delta'$.


\begin{remark}
In Lemma \ref{l.lcon}  we do not assume that $\overline{U_{x, y, \varepsilon}}$ and $\overline{U_{x', y', \varepsilon'}}$ are disjoint. For example, it is possible that $x=x'$.
\end{remark}
\begin{proof}[Proof of Lemma \ref{l.lcon}]
For given small $\delta, \delta'$ there exists an affine invertible orientation preserving map $L'$ such that the line $\gamma_{x', y'}$ is a line of fixed points for $L'$, $L'\circ L(U_{x, y, \delta})=U_{x', y', \delta'}$, and the composition $L'\circ L$ sends subspaces orthogonal to $\gamma_{x, y}$ to subspaces orthogonal to $\gamma_{x', y'}$. The $C^0$-norm of restriction $L'|_{L(U_{x, y, \delta})}$ tends to zero as $\delta, \delta'\to 0$. Moreover, the restriction $L'|_{L(U_{x, y, \delta})}$ can be extended to a homeomorphism $H:\overline{U_{x', y', \varepsilon'}}\to \overline{U_{x', y', \varepsilon'}}$ such that $H|_{L(U_{x, y, \delta})}=L'|_{L(U_{x, y, \delta})}$, and $\|H\|_{C^0}\to 0$ as $\delta, \delta'\to 0$.  The composition $L'\circ L\circ S^*$ is an affine map that sends the vector field $w$ to a vector field $\tilde w$ in $U_{x', y', \delta'}$ directed towards the line $\gamma_{x', y'}$. The push-forward of the vector field $w'$ under the map $(R^*)^{-1}$ (let us denote it by $\tilde w'$) is also a vector field in $U_{x', y', \delta'}$ directed towards the line $\gamma_{x', y'}$. The standard fundamental domain arguments show that the time-1 shift along the flows defined by $\tilde w$ and $\tilde w'$ are conjugate, and $C^0$-norm of the conjugacy is not greater than $\delta'$. Let us denote that conjugacy by $h^*$. Then $\hat h=h^*\circ L'$ satisfies all the conditions of Lemma \ref{l.lcon} if $\delta, \delta'$ are sufficiently small.
\end{proof}
Finally, there exists a collection of sufficiently small numbers $\delta_{k,i}>0$ and $\delta_i>0$  such that setting \hyperlink{Si}{$S_i$ to be the $S$ constructed with $\delta_i$ and $\delta_{k,i}$}, the restriction of $h^{\uparrow}_i\circ h_i^{\rightarrow}\circ h^{\uparrow}_i\circ S_{\delta_{k,i}}$ to $U_{x_k, x_{n_i}, \varepsilon_{k,i}}$ is an affine map. An application of Lemma \ref{l.lcon} to these choices finishes the construction of $\hat h_i$.

\vspace{5pt}

\noindent{\bf Definition of $R_i$.}\\
The constructions of $R_i$ is very similar to the construction of $S_i$ described above, except that it is an expansion rather than a contraction.
\vspace{5pt}

Notice that together with Propositions  \ref{p.PosV1to4} and \ref{p.PosW}, the construction of homeomorphisms $\{h_i\}$ guarantees that the following holds:

\begin{coro}\label{cor.Pos}
The sizes of the cigars $\{\varepsilon_{n,m}'\}$ can be chosen small enough to guarantee that if $x\in U_{x_n, x_m, \varepsilon_{n,m}'}$, and $h_i(x)\in U_{x_{\bar n}, x_{\bar m}, \varepsilon_{\bar n, \bar m}'}$, then $|Pos(h_i(x))-Pos(x)|<\frac{1}{500\cdot 2^{i+1}}$. As a result, if an orbit of $x\in U_{x_n, x_m, \varepsilon_{n,m}'}$ is inside of $\cup_{n,m}U_{x_n, x_m, \varepsilon_{n,m}'}$, then the relative position of any point in the orbit of $x$ differs from $Pos(x)$ by at most $\frac{1}{500}$.
\end{coro}

Also, let us notice that due to the construction of $h_i$ presented above we immediately have the following estimate:
\begin{proposition}\label{p.coest}
For each $i\in \mathbb{N}$ we have
$$\text{dist}_{C^0}(id, h_i)\le 2\max_{k\ne n_i, m_i}(\varepsilon_{k, n_i}, \varepsilon_{k, m_i})+2\varepsilon_{n_i, m_i}+6d_i$$
\end{proposition}
Combined with Lemma \ref{l.sumofdi},  Proposition \ref{p.coest}  gives the following:
\begin{coro}\label{cor.conv}
$$
\sum_{i=1}^\infty\text{dist}_{C^0}(id, h_i)<\infty
$$
\end{coro}
\pf
Let us recall that we chose $\{\varepsilon_{m,n}\}$ such  that $\varepsilon_{m,n}\le 2^{-m-n}$. This implies that $\sum_{i=1}^\infty \max_{k\ne n_i, m_i}(\varepsilon_{k, n_i}, \varepsilon_{k, m_i})<\infty$ and $\sum_{i=1}^\infty \varepsilon_{n_i, m_i}<\infty$.
Combined with Lemma \ref{l.sumofdi} that claims that $\sum_{i=1}^\infty d_i<\infty$, it completes the proof of Corollary \ref{cor.conv}.
\qed

Let us recall that the sets $Base_i$ were defined in (\ref{e.Basei}) as
$$
Base_i=\left(\bigcup_{n\in \mathbb{N}}B_{x_n}\right)\bigcup \left(\bigcup_{n\ne m}U_{x_n, x_m, \varepsilon_{n,m}}\right)\bigcup \left(\bigcup_{n\ \text{is not frozen at stage $i$}
}U_{r, x_n, \theta_n}\right)
$$

\begin{lemma}\label{l.supphi}
If $h_i(x)\ne x$, then either $h_j(h_i(x))=h_i(x)$ for all $j>i$, or $h_i(x)\in Base_i$.
\end{lemma}
\begin{proof}[Proof of Lemma \ref{l.supphi}.]
Suppose $h_i(x)\ne x$, then $x, h_i(x)\in \text{supp}\,h_i$, i.e.
$$
h_i(x)\in Base_i\cup\left(B_{Q_i}\cup \left(\cup_{k\ne m_i, n_i}V_{k,i}\right)\cup W_i \cup V_{r,i} \right).
$$
Due to our choice of $B_{Q_i}, \{V_{k,i}\}, W_i$, and $V_{r,i}$ defined in Section \ref{ss.BVW} in the case $h_i(x)\not \in Base_i$ the projection $\pi_2(h_i(x))$ is in a \hyperlink{conesK}{cone $K_i$} around $w_i$ in $\mathbb{E}_2$ that is disjoint from cones around $w_j$, $j\ne i$, and hence $h_i(x)\not \in \text{supp}\, h_j$ for any $j>i$.

\end{proof}

Define for each $n\in \mathbb{N}$ the set $S(n)$ by
\begin{multline*}
  S(n)=\left\{k\in \mathbb{N}\ |\ k>n \ \text{and}\ \forall i\in \mathbb{N}, \text{if $h_i(x_k)\ne x_k$,} \right.\\
   \left.\text{ then for any $t\in Cont_m$ with $m\le n$ the vertex $x_t$ is frozen at stage $i$}\right\}
\end{multline*}

For each $n\in \mathbb{N}$ (that will correspond to a vertex $x_n$) and $i\in \mathbb{N}$ (that will correspond to a homeomorphism $h_i$) define a "tail" $T_i(n)\subset \mathbb{N}$ in the following way:

\vspace{4pt}

$\bullet$  $T_0(n)=S(n)$ for each $n\in \mathbb{N}$;

\vspace{4pt}

$\bullet$ $T_{i+1}(n)=T_i(n)$ if $\varphi_i(n)=n$;

\vspace{4pt}

$\bullet$ $T_{i+1}(n)=T_{i+1}(m)=T_i(n)\cap T_i(m)$, if $\varphi_i$ transposes $n$ and $m$, $n\ne m$.

 \vspace{4pt}

For each $n\in \mathbb{N}$ and $i\in \mathbb{N}\cup \{0\}$ define a set (that we will informally call a ``comb")

$$
\mathcal{B}_i(n)=U_{r, x_n, \theta_n}\bigcup \left(\bigcup_{k\in T_i(n)} U_{x_n, x_k, \varepsilon_{n,m}'}\right)
$$

\begin{lemma}
If  $n_1\ne n_2$, then for any $i, j\in \mathbb{N}$ we have $\mathcal{B}_i(n_1)\cap \mathcal{B}_j(n_2)=\emptyset$.
\end{lemma}

\begin{definition}
Let us say that a cigar $U_{x_m, x_n, \varepsilon_{n,m}'}$ is $i$-relaxed if it is not a part of $\mathcal{B}_i(k)$ for any $k$, or if it is a part of $\mathcal{B}_i(k)$ with $k$ being frozen at stage $i$.
\end{definition}

The next three statements directly follow from the definitions:

\begin{lemma}
If  $n_1\ne n_2$, then for any $i, j\in \mathbb{N}$ we have $\mathcal{B}_i(n_1)\cap \mathcal{B}_j(n_2)=\emptyset$.
\end{lemma}

\begin{lemma}
If $U_{x_m, x_n, \varepsilon_{n,m}'}$ is $i$-relaxed, then it is also $j$-relaxed for any $j>i$.
\end{lemma}

\begin{lemma}
If $U_{x_m, x_n, \varepsilon_{n,m}'}$ is $i$-relaxed, then $U_{h_i(x_m), h_i(x_n), \varepsilon_{\varphi_i(m),\varphi_i(n)}'}$  is also $i$-relaxed.
\end{lemma}

Due to Remark \ref{r.mnmoving} we get the following statement:
\begin{lemma}\label{l.57}
If $n$ is not frozen at stage $i$, and $h_i(x_n)=x_n$, then $\mathcal{B}_i(n)\subseteq \text{Fix}\,(h_i)$.
\end{lemma}

It will be convenient to use the following terminology:

\begin{definition}
A point $x_n$ is frozen at stage $i$, if $n$ is frozen at stage $i$. 

\vspace{8pt}

A cigar $U_{x_n, x_m, \varepsilon_{n,m}}$ is frozen at stage $i$ if both $n$ and $m$ are  frozen at stage $i$. .

\vspace{8pt}

A comb $\mathcal{B}_i(n)$ is frozen at stage $i$, if $n$ is frozen at stage $i$. 

\end{definition}

\begin{definition}
Given a point $x\in \mathbb{R}^5$ and $i\in \mathbb{N}$, denote by $O_i(x)$ an orbit of the point $x$ starting with the homeomorphism $h_i$, i.e.
$$
O_i(x)=\{x, h_i(x), h_{i+1}\circ h_i(x), h_{i+2}\circ h_{i+1}\circ h_i(x), \ldots\}
$$
\end{definition}

\begin{lemma}\label{l.fball}
Suppose a vertex $x_k$ is frozen at stage $i$ (i.e. $\varphi_j(k)=k$ for all $j\ge i$). Then for any point $x\in B_{x_k}$ an orbit of the point $x$ starting with the homeomorphism $h_i$ (i.e. $O_i(x)$) is finite.
\end{lemma}
\begin{proof}[Proof of Lemma \ref{l.fball}.]
Let us recall that by the choice of the size of the ball $B_{x_k}$ (see Remark \ref{r.35}), for any point $x\in B_{x_k}\cap U_{x_k, x_n, \varepsilon_{m,n}'}$ we have $Pos(x)<\gamma$. Denote by $\frak{F}_k$ the set
$$
\frak{F}_k=B_{x_k}\cup\left(\cup_{n\ne k}\left\{x\in U_{x_k, x_n, \varepsilon_{m,n}'}, Pos(x)<2\gamma\right\}\right),
$$
where $\gamma$ is the constant from Corollary \ref{cor.Pos}. Due to Lemma \ref{l.supphi}, if $x\in B_{x_k}$, and $h_i(x), h_{i+1}(x), \ldots, h_{i+m}(x)\in Base_i$, then $h_{i+m}(x)\in \frak{F}_k$. Also, if $x\in B_{x_k}$, and $h_j(x)\ne x$ for some $j\ge i$, then either $h_j(x)\in \frak{F}_k$, or $h_l(h_j(x))=h_j(x)$ for all $l>j$. Indeed, this follows from the assumption that $x_k$ is frozen at stage $i$ and Proposition \ref{p.support}. Therefore, we just need to show that it is impossible for a point $x\in B_{x_k}$ to have an infinite orbit inside of $\frak{F}_k$. But if $x\in U_{x_k, x_n, \varepsilon_{k,n}'}$, then this follows from the fact that $Cont_n$ is a finite set.
\end{proof}

\begin{lemma}\label{l.ball}
For any vertex $x_n$, any point $x\in B_{x_n}$, and any $i\in \mathbb{N}$, an orbit of the point $x$ starting with the homeomorphism $h_i$ (i.e. $O_i(x)$) is finite.
\end{lemma}
\begin{proof}[Proof of Lemma \ref{l.ball}.]
Take any $x_n\in B_{x_n}$. If $n$ is frozen at stage $i$, then Lemma \ref{l.fball} can be applied, and hence $O_i(x)$ is finite. Suppose $n$ is not frozen at stage $i$. Let $l\in \mathbb{N}$ be the smallest index $l\ge i$ such that $h_l(x_n)\ne x_n$. Denote by $\frak{F}_n^*$ the set
$$
\frak{F}_n^*=\frak{F}_n\cup U_{x_n, r, \eta_n}.
$$
There are three possibilities:

\vspace{3pt}

1) $h_l\circ \ldots\circ h_i(x)\not \in Base_i$;

\vspace{3pt}

2) $h_l\circ \ldots\circ h_i(x)\in \frak{F}_{n}^*$;

\vspace{3pt}

2) $h_l\circ \ldots\circ h_i(x)\in  \frak{F}_{\varphi_l(n)}^*$.

\vspace{3pt}

In the first case $h_l\circ \ldots\circ h_i(x)$ is a fixed point of any $h_j$ with $j>l$ (due to Lemma \ref{l.supphi}), and hence $O_i(x)$ is finite.

In the second case, if $x_n$ is frozen at stage $l+1$, then the same arguments that were used in the proof of Lemma \ref{l.fball} show that $O_i(x)$ is finite.

In the third case, if $x_{\varphi_l(n)}$ is frozen at stage $l+1$, then, once again, the same arguments that were used in the proof of Lemma \ref{l.fball} show that $O_i(x)$ is finite.

In the second case, if $x_n$ is not frozen at stage $l+1$, let $l'\in \mathbb{N}$ be the smallest index $l'>l$ such that $h_{l'}(x_n)\ne x_n$, and proceed with the same three cases as above.

Similarly, in the third case, if $x_{\varphi_l(n)}$ is not frozen at stage $l+1$, let $l'\in \mathbb{N}$ be the smallest index $l'>l$ such that $h_{l'}(x_n)\ne x_n$, and proceed with the same three cases as above.

Since $Cont_n$ is finite, this process will stop after finitely many steps, hence $O_i(x)$ must be finite.
\end{proof}

\begin{lemma}\label{l.cigarfinite}
If $U_{x_m, x_n, \varepsilon_{m,n}'}$ is $i$-relaxed, then for any point $x\in U_{x_m, x_n, \varepsilon_{m,n}'}$,  an orbit of the point $x$ starting with the homeomorphism $h_i$ (i.e. $O_i(x)$) is finite.
\end{lemma}
\begin{proof}[Proof of Lemma \ref{l.cigarfinite}.]
Without loss of generality we can assume that $x\in U_{x_m, x_n, \varepsilon_{m,n}'}\backslash (B_{x_m}\cup B_{x_n})$, since otherwise we can refer to Lemma  \ref{l.ball}. Let $l\ge i$ be the smallest index such that $h_l(x_n)\ne x_n$ or $h_l(x_m)\ne x_m$. Then one of the following cases must hold:

\vspace{4pt}

1) $h_l(x)\not\in Base_l$;

\vspace{4pt}

2) $h_l(x)\in B_{x_m}\cup B_{x_n}\cup B_{h_l(x_n)}\cup B_{h_l(x_m)}$;

\vspace{4pt}

3)  $h_l(x)\in U_{x_m, x_n, \varepsilon_{m,n}'}\backslash (B_{x_m}\cup B_{x_n})$;

\vspace{4pt}

4)  $h_l(x)\in U_{h_l(x_m), h_l(x_n), \varepsilon_{\varphi_i(m), \varphi_i(n)}'}\backslash  (B_{h_l(x_n)}\cup B_{h_l(x_m)})$.

\vspace{4pt}

In the first case the orbit $O_i(x)$ is finite by Lemma \ref{l.supphi}.

In the second case the orbit $O_i(x)$ is finite by Lemma \ref{l.ball}.

In the third or fourth case consider the smallest index  $l'>l$ such that $h_{l'}(h_l(x))\ne h_l(x)$. Then consider the four cases similar to those listed above, etc. Since $Cont(m,n)$ is finite, this process will stop after finitely many steps, which proves Lemma \ref{l.cigarfinite}.
\end{proof}

\begin{lemma}\label{l.postmiss}
If $x\in U_{x_n, r, \eta_n}$, and $x_n$ is frozen at stage $i$, then the orbit $O_i(x)$ is finite.
\end{lemma}
\begin{proof}[Proof of Lemma \ref{l.postmiss}]
If $x_n$ is frozen at stage $i$, $x\in U_{x_n, r, \theta_n}$, and $l\ge i$ is the smallest index such that $h_l(x)\ne x$, then $x$ must belong to $V_{x_n, x_s, x_t, \varepsilon_{s,t}^{final}}$ for some $s, t\in \mathbb{N}$, where either $s$ or $t$ belongs to $T_i(n)$. Without loss of generality assume that $s\in T_i(n)$. If $h_l(x)\not \in Base_i$, then $O_i(x)$ is finite by Lemma \ref{l.supphi}. If $h_l(x)\in Base_l$, then either $h_l(x)\in B_{x_n}$ (and then $O_i(x)$ is finite by Lemma \ref{l.ball}), or $h_l(x)\in U_{x_n, x_s, \varepsilon_{n,s}'}\cup U_{x_n, x_t, \varepsilon_{n,t}'}$. Notice that both $U_{x_n, x_s, \varepsilon_{n,s}'}$ and $U_{x_n, x_t, \varepsilon_{n,t}'}$ are $l$-relaxed. Indeed, $U_{x_n, x_s, \varepsilon_{n,s}'}$ is $l$-relaxed by definition, since $x_n$ is frozen at stage $l$. Let us show that $U_{x_n, x_t, \varepsilon_{n,t}'}$ is $l$-relaxed. If not, it must be either a part of $\mathcal{B}_l(m)$ for some $m\ne t$, or a part of $\mathcal{B}_l(t)$. If $U_{x_n, x_t, \varepsilon_{n,t}'}$ is a part of $\mathcal{B}_l(m)$ for some $m\ne t$, then $x_m$ is frozen at stage $l$ (otherwise $h_l$ cannot move $x_t$), and then $U_{x_n, x_t, \varepsilon_{n,t}'}$ is $l$-relaxed by definition. And if $U_{x_n, x_t, \varepsilon_{n,t}'}$ is a part of  $\mathcal{B}_l(t)$, then $n>t$, and hence $h_l(s)=s$ if $t$ is not frozen at stage $l$, contradiction. Therefore, both $U_{x_n, x_s, \varepsilon_{n,s}'}$ and $U_{x_n, x_t, \varepsilon_{n,t}'}$ are $l$-relaxed, and by Lemma \ref{l.cigarfinite} the orbit $O_i(x)$ must be finite.
\end{proof}

\begin{lemma}\label{l.comb}
If $x\in \mathcal{B}_i(n)$, and then the orbit $O_i(x)$ is finite.
\end{lemma}
\begin{proof}[Proof of Lemma \ref{l.comb}]
If $n$ is frozen at stage $i$, then either $x$ belongs to some $U_{x_n, x_m, \varepsilon_{n,m}'}$ which is $i$-relaxed (and then $O_i(x)$ is finite due to Lemma \ref{l.cigarfinite}), or $x$ belongs to $U_{x_n, r, \theta_n}$ (and $O_i(x)$ is finite by Lemma \ref{l.postmiss}).

Suppose that $n$ is not frozen at stage $i$. Let $l\ge i$ be the smallest index such that $h_l(x_n)\ne x_n$. Notice that for any $j<l$ we have $h_j(x)=x$ (due to Lemma \ref{l.57}). One of the following cases must hold:

\vspace{4pt}

1) $h_l(x)\not\in Base_l$;

\vspace{4pt}

2) $x\not \in \mathcal{B}_l(n)$ (which is possible since it is possible that $\mathcal{B}_l(n)\ne \mathcal{B}_i(n)$);

\vspace{4pt}

3) $x \in \mathcal{B}_l(n)$, and $n$ is frozen at stage $l+1$;

\vspace{4pt}

4) $x \in \mathcal{B}_l(n)$, and $n$ is not frozen at stage $l+1$;

\vspace{4pt}

5) $x \in \mathcal{B}_l(\varphi_l(n))$, and $\varphi_l(n)$ is frozen at stage $l+1$;

\vspace{4pt}

6) $x \in \mathcal{B}_l(\varphi_l(n))$, and $\varphi_l(n)$ is not frozen at stage $l+1$.

\vspace{4pt}

In the case 1), the orbit $O_i(x)$ is finite due to Lemma \ref{l.supphi}.

In the case 2), the orbit of $O_i(x)$ is finite due to Lemma \ref{l.cigarfinite}.

In the case 3) or 5), the orbit of $O_i(x)$ is finite due to Lemma \ref{l.cigarfinite} and Lemma \ref{l.postmiss}.

In the case 4) or 6), take the smallest index $l'>l$ such that $h_{l'}(x_n)\ne x_n$ (in the case 4)) or $h_{l'}(x_{\varphi_l(n)})\ne x_{\varphi_l(n)}$ (in the case 6)), and consider the same 6 cases. Since $Cont_n$ is finite, this process will end in a finite number of steps, hence $O_i(x)$ must be finite.
\end{proof}

Combination of  Lemma \ref{l.ball}, Lemma \ref{l.cigarfinite}, and Lemma \ref{l.comb} now shows that an orbit of any point is finite, which completes the proof of Proposition \ref{p.key}, and hence of Theorem \ref{t.main}.

\section{Open questions and conjectures}\label{s.questions}

Our results provoke a long list of farther questions. Let us briefly list some of them here.  More appear in \cite{classDS}.

\begin{enumerate}

\item Let $M$ be a compact manifold without boundary of dimension at least two. Is there a dense $\mathcal G_\delta$ subset of 
$\text{Diff}^\infty(M)$ such that the equivalence relation of topological conjugacy is Borel? 
\smallskip 

One candidate would be the Kupka-Smale diffeomorphisms. (Recall that a diffeomorphism   is Kupka-Smale if every periodic orbit is hyperbolic, and stable and unstable manifolds of any two periodic orbits are transversal. It is known that the set of Kupka-Smale diffeomorphisms of a given manifold $M$ forms a dense $G_{\delta}$ subset of $\text{Diff}^r(M)$ for each $r=1, 2, \ldots, \infty$, see \cite{Kup} and \cite{S1963}.)

\item\label{q.MS} Is the equivalence relation given by existence of a topological conjugacy on the boundary of the set of Morse-Smale diffeomorphisms of a given manifold $M$      Borel\footnote{After this text was written, we realized that the answer to this question is ``no"; we intend to provide a proof of this statement in the final version of the paper.}? (Morse-Smale diffeomorphisms are defined in \ref{d.MS}) 
     
\item Is the collection of diffeomorphisms $f$ in $\text{Diff}^\infty(M)$   that are topologically conjugate to a structurally stable diffeomorphism a Borel set?

\item Let $X$ be the collection of diffeomorphisms of a manifold $M$. Is the equivalence relation of \emph{conjugacy by homeomorphisms} given by a Polish group action? Reducible to an $S^\infty$-action?

\end{enumerate}

\appendix
\section{Locations of notational definitions}
This list gives the notation and the first page it appears on. 
\begin{description}
	\item[Page 21] 
	\[I_{m,n}, \ \ \  {\mathbf{I}_n}\]
	\item[Page 22]
	\[\mathbf{I}=\bigcup_{n}\mathbf{I}_{n},\ \ \ \ \alpha_{m,n}\] 
	\item[Page 23]
	\[\gamma_{x,y}\]
	\item[Page 24]
	\[\Pi_{x,y}(t), \ \ \ \ \gamma_{m,n}, \ \ \ \ U_{x,y,\epsilon},\ \ \ \  U_{m,n, \epsilon},\]
	\[\text{Cigars},\ \ \ \ \mathbf{U}_{m,n}, \ \ \ \ \mathbf{U}'_{m,n},\ \ \ \ \mathbf{U}''_{m,n} 	\]
	\[ \epsilon_{m,n},\ \ \ \ \epsilon'_{m,n}, \ \ \ \ \epsilon''_{m,n}\]
	\item[Page 25]\[f_{m,n}^+,\ \ \ \  f_{m,n}^-,\ \ \ \  \lambda_{m,n}, \ \ \ \ w_{m,n},\ \ \ \ \Theta,\ \ \ \ v_{m,n}\]
	\item[Page 28] \[Cont_m\]
	\item[Page 29]\[Cont(n,m)\]
	\item[Page 30]\[\text{frozen}, \ \ \ \ J(n), \ \ \ \ T(n),\ \ \ \ d_i\]
	\item[Page 31]\[E_1, \ \ \ \ E_2, \ \ \ \ E^{(i)}, \ \ \ \ h_i\]
	\item[Page 32]\[\rho_n, \ \ \ \ w_i,\ \ \ \ \ q_i, \ \ \ \ \beta_j\]
	\[ K_j,\ \ \ \ Q_i\ \ \ \ B_{Q_i}\]
	\item[Page 33]
	\[Q^u_i,\ \ \ \ Q^d_i, \ \ \ \ \delta_i, \ \ \ \ v=v(x,y, w),\ \ \ \ g^t_v\]
	\[\psi=\psi(x,y,\epsilon, w),\ \ \ \ \Delta=\Delta(x,y, \epsilon, w)\]
	\item[Page 35]\[V(z,x,y,\xi,w),\ \ \ \ v_{k,i},\ \ \ \  W(x,y,\xi,w)\]
	\[V_{k,i},\ \ \ \ W_{i}\]
	\item[Page 36]\[Pos(z)\]
	\item[Page 42]\[\tilde{\theta}_n\]
	\item[Page 43] \[Base_i, \ \ \ \ \Psi,\ \ \ \  \mathfrak{k}\]
	\[ R_i,\ \ \ \ \hat h_i,\ \ \ \  h^{\uparrow}_i,\ \ \ \  h^{\rightarrow}_i,\ \ \ \  h_i^{\uparrow},\ \ \ \  S_i\]
	\item[Page 46] \[\Xi,\ \ \ \  \Xi'\]
	\item[Page 48]\[T_i(n),\ \ \ \  i\text{-relaxed}\]
	\item[Page 49]\[\frak{F}, \ \ \ \ \frak{F}^*\]

\end{description}

\end{document}